\documentclass[12pt]{article}
\usepackage[dvipdfm, a4paper, margin=2cm]{geometry}
\usepackage{amsmath}%
\usepackage{amssymb}%
\usepackage{amsthm}
\usepackage{amsfonts}
\usepackage{color}
\usepackage{eufrak}
\allowdisplaybreaks[4]
\newtheorem{theorem}{Theorem}[section]
\newtheorem{proposition}[theorem]{Proposition}
\newtheorem{corollary}[theorem]{Corollary}
\newtheorem{lemma}[theorem]{Lemma}
\theoremstyle{definition}
\newtheorem{remark}[theorem]{Remark}

\newtheorem{definition}[theorem]{Definition}

\makeatletter \@addtoreset{equation}{section}

\makeatother \makeatletter

\DeclareMathOperator{\End}{End}

\DeclareMathOperator{\Hom}{Hom}

\newcommand{\BF}{\mathbb{F}}
\newcommand{\BN}{\mathbb{N}}
\newcommand{\BZ}{\mathbb{Z}}
\newcommand{\BC}{\mathbb{C}}

\newcommand{\fg}{\mathfrak{g}}
\newcommand{\fh}{\mathfrak{h}}
\newcommand{\D}{\mathcal{D}}
\newcommand{\B}{\mathcal{B}}

\newcommand{\bk}{\mathbf{k}}

\newcommand{\1}{\mathbf{1}}
\newcommand{\+}[1]{\langle#1\rangle}

\newcommand{\del}[2]{\delta\left(\frac{#1}{#2}\right)}
\makeatletter
\newtoks\@enLab  
\def\@enQmark{?}
\def\@enLabel#1#2{%
  \edef\@enThe{\noexpand#1{\@enumctr}}%
  \@enLab\expandafter{\the\@enLab\csname the\@enumctr\endcsname}%
  \@enloop}
\def\@enSpace{\afterassignment\@enSp@ce\let\@tempa= }
\def\@enSp@ce{\@enLab\expandafter{\the\@enLab\space}\@enloop}
\def\@enGroup#1{\@enLab\expandafter{\the\@enLab{#1}}\@enloop}
\def\@enOther#1{\@enLab\expandafter{\the\@enLab#1}\@enloop}
\def\@enloop{\futurelet\@entemp\@enloop@}
\def\@enloop@{%
  \ifx A\@entemp         \def\@tempa{\@enLabel\Alph  }\else
  \ifx a\@entemp         \def\@tempa{\@enLabel\alph  }\else
  \ifx i\@entemp         \def\@tempa{\@enLabel\roman }\else
  \ifx I\@entemp         \def\@tempa{\@enLabel\Roman }\else
  \ifx 1\@entemp         \def\@tempa{\@enLabel\arabic}\else
  \ifx \@sptoken\@entemp \let\@tempa\@enSpace         \else
  \ifx \bgroup\@entemp   \let\@tempa\@enGroup         \else
  \ifx \@enum@\@entemp   \let\@tempa\@gobble          \else
                         \let\@tempa\@enOther
             \fi\fi\fi\fi\fi\fi\fi\fi
  \@tempa}
\newlength{\@sep} \newlength{\@@sep}
\providecommand{\sfbc}{\rmfamily\upshape}
\providecommand{\sfn}{\rmfamily\upshape}
\def\@enfont{\ifnum \@enumdepth >1\let\@nxt\sfn \else\let\@nxt\sfbc \fi\@nxt}
\def\enumerate{%
   \ifnum \@enumdepth >3 \@toodeep\else
      \advance\@enumdepth \@ne
      \edef\@enumctr{enum\romannumeral\the\@enumdepth}\fi
   \@ifnextchar[{\@@enum@}{\@enum@}}
\def\@@enum@[#1]{%
  \@enLab{}\let\@enThe\@enQmark
  \@enloop#1\@enum@
  \ifx\@enThe\@enQmark\@warning{The counter will not be printed.%
   ^^J\space\@spaces\@spaces\@spaces The label is: \the\@enLab}\fi
  \expandafter\edef\csname label\@enumctr\endcsname{\the\@enLab}%
  \expandafter\let\csname the\@enumctr\endcsname\@enThe
  \csname c@\@enumctr\endcsname7
  \expandafter\settowidth
            \csname leftmargin\romannumeral\@enumdepth\endcsname
            {\the\@enLab\hskip\labelsep}%
  \@enum@}
\def\@enum@{\list{{\@enfont\csname label\@enumctr\endcsname}}%
           {\usecounter{\@enumctr}\def\makelabel##1{\hss\llap{##1}}%
     \ifnum \@enumdepth>1\setlength{\topsep}{\@@sep}\else
           \setlength{\topsep}{\@sep}\fi
     \ifnum \@enumdepth>1\setlength{\itemsep}{0pt plus1pt minus1pt}%
      \else \setlength{\itemsep}{\@@sep}\fi
     \setlength{\parsep}{0pt plus1pt minus1pt}%
     \setlength{\parskip}{0pt plus1pt minus1pt}
                   }}

\def\endenumerate{\par\ifnum \@enumdepth >1\addvspace{\@@sep}\else
           \addvspace{\@sep}\fi \endlist}
\makeatother                                                            %

\begin{document}
\title{Heisenberg VOAs over Fields of Prime Characteristic and Their Representations}
\author{Haisheng Li\footnote{Partially supported by China NSF grant No. 11471268}\\
Department of Mathematical Sciences, Rutgers University,\\
Camden, NJ 08102, USA, and\\
School of  Mathematical Sciences, Xiamen University, Xiamen 361005, China\\
\textit{E-mail address:} \texttt{hli@camden.rutgers.edu}\and
Qiang Mu\\
School of Mathematical Sciences, Harbin Normal University,\\
Harbin, Heilongjiang 150080, China\\
\textit{E-mail address:} \texttt{qmu520@gmail.com}}
\maketitle

\begin{abstract}
In this paper, we study Heisenberg vertex algebras over fields of prime characteristic.
The new feature is that the Heisenberg vertex algebras are no longer simple unlike in the case of characteristic zero.
We then study a family of simple quotient vertex algebras and
we show that for each such simple quotient vertex algebra,
irreducible modules are unique up to isomorphism and every module is completely reducible.
To achieve our goal, we also establish a complete reducibility theorem for a certain category of modules
over Heisenberg algebras.
\end{abstract}

\section{Introduction}

Vertex (operator) algebras  have been extensively studied for about three decades since they were introduced
in the late 1980's (see \cite{B86}, \cite{FLM}, \cite{FHL}). In the past,  studies have been mainly focused on vertex algebras
over the field of complex numbers, or sometimes over a general field of characteristic zero.
(Note that the notion of vertex algebra, introduced by Borcherds in \cite{B86}, is over an arbitrary field.)
Vertex algebra has deep connections with numerous fields in mathematics and physics,
and  it has been proved to be an interesting and fruitful  research field.
Among the important examples of vertex operator algebras are those associated to
infinite-dimensional Lie algebras such as affine  Lie algebras including Heisenberg algebras
and the Virasoro algebra
(see \cite{FZ}), and those associated to positive definite even lattices (see \cite{FLM}).

Recently,  Dong and Ren in a series of papers studied vertex algebras over an arbitrary field, in particular over a field of prime characteristic. More specifically, Dong and Ren studied the representations for a general vertex algebra over an
arbitrary field in \cite{DR1} and they studied vertex algebras associated to the Virasoro algebra in \cite{DR2}.
Previously, Dong and Griess (see \cite{DG}) studied integral forms of vertex algebras (over a field of characteristic zero).
Based on \cite{DG}, Mu (see \cite{M}) studied vertex algebras over fields of prime characteristic
obtained from integral forms of lattice vertex operator algebras.

In this paper, we study Heisenberg vertex algebras and their representations over a general field
of  prime characteristic. Though many results for characteristic zero still hold for prime characteristic,
there are interesting new features. For example, for a nonzero level, Heisenberg vertex algebras are no longer simple,
unlike in the case of characteristic zero.
We then study a family of simple quotient vertex algebras.
We show that for each such simple quotient vertex algebra,
irreducible modules are unique up to isomorphism and every module is completely reducible.
To achieve this we obtain a complete reducibility theorem for a category of affine Heisenberg Lie algebras.

We now give a more detailed count of the contents of this paper.
Let $\BF$ be any field and let $\fh$ be a finite-dimensional vector space over $\BF$
equipped with a non-degenerate symmetric bilinear form
$\langle\cdot,\cdot\rangle$. Associated to the pair $(\fh,\langle\cdot,\cdot\rangle)$,
one has an affine Lie algebra $\widehat{\fh}=\fh\otimes \BF[t,t^{-1}]\oplus \BF \bk$, where $\bk$ is central and
$$[\alpha\otimes t^{m},\beta\otimes t^{n}]=m\langle\alpha,\beta\rangle \delta_{m+n,0}\bk$$
for $\alpha,\beta\in \fh,\ m,n\in \BZ$. For any $\ell\in \BF$, one has a vertex algebra $V_{\widehat{\fh}}(\ell,0)$
whose underlying vector space is
the $\widehat{\fh}$-module generated by a distinguished vector ${\bf 1}$, subject to relations
$$\bk \cdot {\bf 1}=\ell{\bf 1}\   \   \mbox{ and }\   \  (\fh\otimes \BF[t]){\bf 1}=0.$$
It was known (cf. \cite{LL}) that for $\BF=\BC$,
$V_{\widehat{\fh}}(\ell,0)$ is a simple vertex operator algebra for every nonzero $\ell$ and these are all isomorphic.
As for the Lie algebra $\widehat{\fh}$ with $\BF = \BC$,
the subalgebra $\widehat{\fh}'=\sum_{n\ne 0}\fh\otimes \BF t^{n}+ \BF \bk$ is a Heisenberg algebra and
 there is an important complete reducibility theorem for a certain category of
$\widehat{\fh}'$-modules (see \cite{lw2}, \cite{kac}). This complete reducibility theorem is very useful in studying
the representations of affine Kac-Moody algebras (see \cite{lw2}) and in studying the vertex algebras and their
 representations associated to even lattices (see \cite{FLM}, \cite{dong93}, \cite{dong94}).

For a prime characteristic, however, the situation is quite different. One difference is that
the subalgebra $\widehat{\fh}'=\sum_{n\ne 0}\fh\otimes \BF t^{n}+ \BF \bk$
 has an infinite-dimensional center, which is no longer a Heisenberg algebra.
 (Note that if ${\rm char}\;\BF =p$, $\fh\otimes \BF [t^{p},t^{-p}]$  is contained in the center of $\widehat{\fh}$.)
Due to this, vertex algebra $V_{\widehat{\fh}}(\ell,0)$ contains infinitely many maximal ideals.
Assume ${\rm char}\;\BF=p$ (a prime). Set $\widehat{\fh}_{+}=\fh\otimes t^{-1}\BF[t^{-1}]$.
For $u\in \fh,\ n\in \BZ$, we write $u(n)$ for $u\otimes t^{n}$ as usual.
Let $\lambda\in (\widehat{\fh}_{+})^{*}$. Denote by $J(\lambda)$ the $\widehat{\fh}$-submodule of $V_{\widehat{\fh}}(\ell,0)$, generated by vectors
$$(u(-np)-\lambda(u(-np))){\bf 1},\   \   \left(u(-n)-\lambda(u(-n))\right)^{p}{\bf 1}$$
for $u\in \fh,\ n\ge 1$. It is proved in this paper that $J(\lambda)$ is a maximal $\widehat{\fh}$-submodule of
$V_{\widehat{\fh}}(\ell,0)$ and
these exhaust all the maximal $\widehat{\fh}$-submodules.
Furthermore, it is proved that $J(\lambda)$ is an ideal of vertex algebra $V_{\widehat{\fh}}(\ell,0)$
if and only if $\lambda(u(-n))=0$ for all $u\in \fh,\ n\ge 2$. Then for any such $\lambda$, the quotient of
$V_{\widehat{\fh}}(\ell,0)$ by $J(\lambda)$ gives us a
simple vertex algebra $L_{\widehat{\fh}}(\ell,0,\lambda)$. As our main results, we show that
every $L_{\widehat{\fh}}(\ell,0,\lambda)$-module
is completely reducible and the adjoint module is the only irreducible module up to equivalence.

On the other hand, set
$$\widehat{\fh}''=\sum_{n\notin p\BZ}\fh\otimes \BF t^{n}+ \BF \bk,$$
which is a Heisenberg algebra.
To achieve our main results, we establish and use a complete reducibility theorem for a certain category of
$\widehat{\fh}''$-modules. We give an analogue of the canonical realization of Heisenberg algebras in terms of differential operators and multiplication operators.

This paper is organized as follows:
In Section 2, we present some basic results on vertex algebras of prime characteristic.
In Section 3, we discuss vertex algebras associated to affine Lie algebras. In Section 4, we study a category of
modules for Heisenberg algebras.
In Section 5, we study a family of simple quotient Heisenberg vertex algebras  and determine their irreducible modules.

\section{Basics}
In this section, we present some basic results on vertex algebras over an arbitrary field, especially over a field of prime characteristic.
The proofs for most of these results are the same as those for characteristic zero, which can be found in \cite{LL} for example.
For completeness, we also provide some expository details.

Let $\BF$ be an arbitrary field. Throughout this paper, $x,y,z,x_0,x_1,x_2,\ldots$ are mutually commuting independent formal variables.
Vector spaces are considered to be over $\BF$. 
If the letter $p$ appears in some expression, by default $\BF$  is assumed to be of characteristic $p$.
We use the usual symbols $\BZ$ for the integers, $\BZ_+$ for the positive integers,
and $\BN$ for the nonnegative integers.
 For $n\in\BZ$ and $k\in \BN$, we consider the binomial coefficient
\begin{equation*}
    \binom{n}{k}=\frac{n(n-1)\cdots(n-k+1)}{k!},
\end{equation*}
which is an integer, as a number in the prime subfield of $\BF$.
Furthermore, for $n\in\BZ$, define $(x+y)^n$  to be the formal series
\begin{equation}\label{ebinom-expansion}
    (x+y)^n=\sum_{k\in\BN}\binom{n}{k}x^{n-k}y^k.
\end{equation}
We shall frequently use Lucas' theorem which states that
for any  $$m=m_0+m_1p+\cdots+m_k p^k,\  \  \  n=n_0+n_1 p+\cdots+n_k p^k
\in\BN,$$
where $p$ is a prime number, $k\in \BN$, $0\le m_i,\  n_i\le p-1$, we have
 \begin{eqnarray}
 \binom{m}{n}\equiv \prod_{i=0}^k \binom{m_i}{n_i} \pmod{p}.
 \end{eqnarray}
In particular, this implies that $\binom{m}{n}=0$ if $m\in p\BZ_+,\ n\notin p\BZ_+$.

For $k\in \BN$, we denote by $\partial_x^{(k)}$ the $k$-th Hasse derivative with respect to $x$, that is,
\begin{equation}
    \partial_x^{(k)}x^m=\binom{m}{k}x^{m-k}
\end{equation}
for $m\in\BZ$. The following relation holds for $m,n\in \BN$:
\begin{equation}\label{eq:partialm}
    \partial_x^{(m)}\partial_x^{(n)}=\binom{m+n}{m}\partial_x^{(m+n)}.
\end{equation}
We formally denote
\begin{equation}
    e^{x_0\partial_x}=\sum_{n\ge0}x_0^n \partial_x^{(n)}.
\end{equation}
Just as in the case of characteristic zero, the following formal Taylor formula holds:
\begin{equation}
    e^{x_0\partial_x}a(x)=a(x+x_0)
\end{equation}
 for $a(x)\in U[[x,x^{-1}]]$, where $U$ is any vector space over $\BF$.

Recall the formal delta-function
$$\delta(x)=\sum_{n\in\BZ}x^n.$$
For $m,n\in\BN$, we have
\begin{equation}\label{eq:partial02}
 \partial_{x_2}^{(n)}x_2^{-1}\del{x_1}{x_2}=(x_1-x_2)^{-n-1}-(-x_2+x_1)^{-n-1}
 =(-1)^n\partial_{x_1}^{(n)}x_2^{-1}\del{x_1}{x_2}.
\end{equation}

The following definition was due to Borcherds (see \cite{B86}):

\begin{definition} \label{def:B86VA}
A \emph{vertex algebra} is a vector space $V$ equipped with a distinguished vector $\1$,
linear operators $\D^{(m)}$ on $V$ for $m\in\BZ$, and bilinear operations
$(u,v)\mapsto u_nv$ from $V\times V$ to $V$ for $n\in\BZ$,
satisfying the following conditions (i)--(v)
for $u,v,w\in V$ and $m,n\in\BZ$:
\begin{enumerate}[(i)]
\item $u_nv=0$ for $n$ sufficiently large.
\item $\1_nv=\delta_{n,-1}v$.
\item $u_n\1=\D^{(-n-1)}u$.
\item $u_nv=\sum_{i\ge0}(-1)^{i+n+1}\D^{(i)}v_{n+i}u$.
\item $(u_mv)_nw=\sum_{i\ge0}(-1)^i\binom{m}{i}(u_{m-i}v_{n+i}w-(-1)^m v_{m+n-i}u_iw)$.
\end{enumerate}
\end{definition}

The following are some (immediate)  consequences (see \cite{LL}, page 90; cf. \cite{B86}):

\begin{lemma}\label{rm:defVA}
Let $V$ be  a vertex algebra. Then
\begin{enumerate}[(1)]
\item $u_{n}\1=0$ for $u\in V,\ n\in \BN$.
\item $\D^{(n)}=0$ for $n<0$.
\item $\D^{(0)}=1_V$, or equivalently, $u_{-1}\1=u$ for $u\in V$.
\item $\D^{(m)}\D^{(n)}=\binom{m+n}{n}\D^{(m+n)}$ for $m,n\in \BN$.
\item $\D^{(n)}(u_{m}v)=\sum_{i=0}^{n}(-1)^{i}\binom{m}{i}u_{m-i}\D^{(n-i)}v$ for $n\in \BN,\ u,v\in V,\ m\in \BZ$.
\end{enumerate}
\end{lemma}


Recall that a {\em coalgebra} is a vector space $C$ equipped with two linear maps
$\Delta: C\rightarrow C\otimes C$, called the {\em co-multiplication}, and $\varepsilon: C\rightarrow \BF$, called
the {\em co-unit}, such that
$$(1\otimes \Delta)\Delta= (\Delta\otimes 1)\Delta, \  \  \  (1\otimes \varepsilon)\Delta=1, \  \  \
(\varepsilon \otimes 1)\Delta=1,$$
(where $1$ is the identity map on $C$).
Furthermore, a {\em bialgebra} is an associative algebra $B$ with unit $1$ which is
a coalgebra at the same time such that the co-multiplication $\Delta$ and co-unit $\varepsilon$
are both algebra homomorphisms.

\begin{remark}\label{rdefinition-B}
Let $\B$ be a vector space with basis $\{\D^{(n)}\mid n\in\BN\}$.
It is well known (cf. \cite{B86}) that $\B$ is a bialgebra with
\begin{eqnarray}
   && \D^{(m)}\cdot  \D^{(n)}=\binom{m+n}{n}\D^{(m+n)}, \quad
    \D^{(0)}=1,\label{eD-multiplication}\\
 &&   \Delta(\D^{(n)})=\sum_{i=0}^n \D^{(n-i)}\otimes \D^{(i)}, \quad
    \varepsilon(\D^{(n)})=\delta_{n,0}\label{ecoproduct-D}
\end{eqnarray}
for $m,n\in\BN$.
\end{remark}

In view of Lemma \ref{rm:defVA}, each vertex algebra $V$ is naturally a $\B$-module.
Following \cite{DR1}, we set
\begin{equation}
    e^{x\D}=\sum_{n\ge0}x^n \D^{(n)}\in \B[[x]],
\end{equation}
so that $e^{x\D}$ acts on every $\B$-module and especially acts on every vertex algebra.
As in the ordinary case, we have
\begin{align*}
 e^{x\D} e^{z\D}=e^{(x+z)\D}\  \   \ \mbox{ and }\  \  \  e^{x\D}e^{-x\D}=1,
\end{align*}
which are due to (\ref{eD-multiplication}), (\ref{ebinom-expansion}), and to the fact that
$\sum_{n=0}^{k}(-1)^{n}\binom{k}{n}=\delta_{k,0}$ for $k\in \BN$.
On the other hand, (\ref{ecoproduct-D}) can be written as
\begin{eqnarray}\label{ecoproduct}
\Delta(e^{x\D})=e^{x\D}\otimes e^{x\D}\   \mbox{ and } \   \varepsilon(e^{x\D})=1.
\end{eqnarray}
(Informally, $e^{x\D}$ is group-like.)
Note that
\begin{equation}
    \partial_x^{(n)} e^{x\D}=\D^{(n)} e^{x\D}\  \  \   \mbox{ for }n\in \BN.
\end{equation}

Using Lemma \ref{rm:defVA} (5) and Definition \ref{def:B86VA} (iv),
as it was mentioned in \cite{DR1}, we immediately have:

\begin{lemma}
Let $V$ be a vertex algebra. For $u\in V$, set
$$Y(u,x)=\sum_{n\in \BZ}u_{n}x^{-n-1}\in (\End V)[[x,x^{-1}]],$$
where $u_{n}$ denote the corresponding linear operators on $V$. Then the following
{\em skew symmetry} and {\em conjugation formula} hold  for $u,v\in V$:
\begin{eqnarray}
&&  Y(u,x)v=e^{x\D}Y(v,-x)u,\label{eskew-symmetry}\\
&&e^{x_0\D}Y(u,x)e^{-x_0\D}=e^{x_0\partial_x}Y(u,x).
\end{eqnarray}
Furthermore,  the following relations hold for $u\in V$:
\begin{eqnarray}
&&Y(v,z){\bf 1}=e^{z\D}v\   \   \   \mbox{ for }v\in V,\\
&&  e^{x_0\D}Y(u,x)e^{-x_0\D}=  Y(e^{x_0\D}u,x)=e^{x_0\partial_x}Y(u,x)=Y(u,x+x_0).
\end{eqnarray}
In particular, we have
\begin{equation}
 [\D^{(1)},Y(u,x)]= Y(\D^{(1)} u,x)=\frac{d}{dx} Y(u,x).
\end{equation}
\end{lemma}

As in the case of characteristic zero, we have (see \cite{LL}, pages 90-91):

\begin{proposition}\label{def:VA}
A vertex algebra can be defined equivalently as a vector space $V$ equipped with a linear map
\begin{equation}
\begin{split}
    Y(\cdot,x):V&\to(\End V)[[x,x^{-1}]]\\
    v&\mapsto Y(v,x)=\sum_{n\in\BZ}v_nx^{-n-1},
\end{split}
\end{equation}
satisfying the following conditions for $u,v\in V$:
\begin{equation}
    u_nv=0\quad \text{for $n$ sufficiently large}
\end{equation}
(the {\em truncation condition}),
\begin{equation}
    Y(\1,x)=1 \quad\text{(the identity operator on $V$)}
\end{equation}
(the {\em vacuum property}),
\begin{equation}\label{ecreation-property}
    Y(v,x)\1\in V[[x]]\quad\text{and}\quad \lim_{x\to0}Y(v,x)\1=v
\end{equation}
(the {\em creation property}),
and
\begin{eqnarray}
    &&x_0^{-1}\del{x_1-x_2}{x_0}Y(u,x_1)Y(v,x_2)-x_0^{-1}\del{x_2-x_1}{-x_0}Y(v,x_2)Y(u,x_1)\nonumber\\
    &&\hspace{3cm}=x_2^{-1}\del{x_1-x_0}{x_2}Y(Y(u,x_0)v,x_2)
\end{eqnarray}
(the {\em Jacobi identity}).
\end{proposition}


\begin{remark}
Let $V$ be a vertex algebra. Then  the following properties hold:
\begin{enumerate}[(i)]
\item {\em Weak commutativity: } For any $u,v\in V$, there exists $k\in \BN$ such that
\begin{equation}
    (x_1-x_2)^k [Y(u,x_1),Y(v,x_2)]=0.
\end{equation}

\item {\em Weak associativity:}
For $u,v,w\in V$, there exists $l\in\BN$ such that
\begin{equation}
    (x_0+x_2)^l Y(Y(u,x_0)v,x_2)w=(x_0+x_2)^l Y(u,x_0+x_2)Y(v,x_2)w.
\end{equation}
\end{enumerate}
\end{remark}

\begin{remark}
Just as in the case with $\BF=\BC$ (see \cite{FHL}, \cite{DL}, \cite{Li96}),
in the definition of a vertex algebra,
the Jacobi identity axiom can be replaced by any one of the following:
\begin{enumerate}[(i)]
\item weak commutativity and weak associativity;
\item weak associativity and skew symmetry;
\item weak commutativity and conjugation formula.
\end{enumerate}
\end{remark}

The notion of module for a vertex algebra $V$ is defined as usual.

\begin{definition}
A {\em $V$-module} is a vector space $W$ equipped with a linear map 
\begin{equation}
\begin{split}
    Y_W(\cdot,x):V&\to(\End W)[[x,x^{-1}]]\\
    v&\mapsto Y_W(v,x)=\sum_{n\in\BZ}v_nx^{-n-1},
\end{split}
\end{equation}
such that for $u,v\in V$ and $w\in W$,
\begin{gather}
    u_nw=0\quad\text{for $n$ sufficiently large},\\
    Y_W(\1,x)=1_{W}\ \  (\mbox{the identity operator on } W),\\
\begin{split}
    &x_0^{-1}\del{x_1-x_2}{x_0}Y_W(u,x_1)Y_W(v,x_2)-x_0^{-1}\del{x_2-x_1}{-x_0}Y_W(v,x_2)Y_W(u,x_1)\\
    &\hspace{3cm}=x_2^{-1}\del{x_1-x_0}{x_2}Y_W(Y(u,x_0)v,x_2).
\end{split}
\end{gather}
\end{definition}

\begin{remark}
In the definition of a $V$-module, the Jacobi identity axiom can be replaced by weak commutativity and weak associativity.
On the other hand, the Jacobi identity axiom can also be equivalently replaced by only weak associativity,
as it was mentioned in \cite{DR1}.
\end{remark}

\begin{definition}
Let $V$ be a vertex algebra. A {\em $(V,\B)$-module} is a $V$-module $(W,Y_W)$ equipped with
a $\B$-module structure such that
\begin{equation*}
    e^{x\D}Y_W(v,z)e^{-x\D}=Y_W(e^{x\D}v,z)\  \  \   \mbox{ for }v\in V.
\end{equation*}
\end{definition}

\begin{remark}\label{th:VB-module}
Let $V$ be a vertex algebra and let $(W,Y_W)$ be a $V$-module.
The same arguments in the case of characteristic zero (cf. \cite{LL}) show that
    \begin{equation}
        Y_W(e^{z_0\D}v,z)=e^{z_0\partial_z}Y_W(v,z)=Y_W(v,z+z_0)\  \   \   \mbox{ for }v\in V.
    \end{equation}
Furthermore, if $(W,Y_W)$ is a $(V,\B)$-module with $V$ a vertex algebra, then we have
\begin{equation}
    e^{x\D}Y_W(v,z)e^{-x\D}=Y_W(e^{x\D}v,z)=e^{x\partial_z}Y_W(v,z)=Y_W(v,z+x)
\end{equation}
for $v\in V$.
\end{remark}

We shall need an analog  of a result of \cite{Li94} (cf. \cite{LL}).

\begin{lemma}\label{vacuum-like}
Let $V$ be a vertex algebra, let $(W,Y_W)$ be a $(V,\B)$-module, and let $w\in W$ be such that
$\D^{(n)}w=0$ for $n\ge 1$.
Then
\begin{equation}\label{eq:WYd-vacuumlike-01}
    Y_W(v,x)w=e^{x\D}v_{-1}w\   \   \   \mbox{ for all }v\in V.
\end{equation}
In particular, $w$ is a vacuum-like vector in the sense that
$v_n w=0$ for all $v\in V,\ n\in \BN$.
\end{lemma}

\begin{proof} Let $v\in V$.  From Lemma~\ref{th:VB-module}, we have
\begin{equation*}
    \binom{-m-1}{n}v_m=\sum_{i=0}^n (-1)^i \D^{(n-i)}v_{m+n}\D^{(i)}
\end{equation*}
for  $m\in\BZ,\  n\in\BN$. 
Then we get
\begin{equation*}
    (-1)^n\binom{m+n}{n}v_mw=\sum_{i=0}^n(-1)^i \D^{(n-i)}v_{m+n}\D^{(i)}w=\D^{(n)}v_{m+n}w.
\end{equation*}
Let $m\in \BN$ be arbitrarily fixed. There exists $n\in\BN$ such that $\binom{m+n}{n}\ne0$ and $v_{m+n}w=0$.
It follows that $v_mw=0$. This proves that $w$ is a vacuum-like vector.
Furthermore, since $\D^{(n)}w=0$ for $n\ge 1$,
we have
\begin{equation*}
    e^{x\D}Y_W(v,x_0)w=e^{x\D}Y_W(v,x_0)e^{-x\D}w=Y_W(v,x_0+x)w.
\end{equation*}
Noticing that $Y_W(v,x_0)w$ involves only nonnegative integer powers of $x_0$,
we can set $x_0$ to zero to obtain \eqref{eq:WYd-vacuumlike-01}.
\end{proof}

Using Lemma \ref{vacuum-like} and the same arguments in \cite{LL} we obtain:

\begin{proposition}\label{pvacuum-like-isomorphism}
Let $(W,Y_W)$ be a $(V,\B)$-module and let $w\in W$ be such that $\D^{(n)}w=0$ for $n\ge 1$.
Then the linear map $f:V\to W$; $v\mapsto v_{-1}w$ is a $V$-module homomorphism. Furthermore, if
$W$ is a faithful $V$-module and if $w$ generates $W$ as a $V$-module, then $f$ is an isomorphism.
\end{proposition}

In the case of characteristic zero, there is a general construction theorem due to
\cite{FKRW}, \cite{MP} (cf. \cite[Theorem~5.7.1]{LL}). For a field $\BF$ of prime characteristic,
by the same proof as in \cite{LL}, we have the following slight modification:

\begin{theorem}\label{th:571}
Let $V$ be a $\B$-module with a distinguished vector $\1$ such that
$\D^{(n)}{\bf 1}=0$ for $n\ge 1$ and
let $T$ be a subset of $V$ equipped with a map
\begin{equation*}
\begin{split}
    Y_0(\cdot,x):T&\to \Hom(V,V((x))),\\
    a&\mapsto Y_0(a,x)=\sum_{n\in\BZ}a_n x^{-n-1}.
\end{split}
\end{equation*}
Assume that the following conditions hold:
\begin{enumerate}[(i)]
\item For $a\in T$, $Y_0(a,x)\1\in V[[x]]$ and $\lim_{x\to 0}Y_0(a,x)\1=a$.
\item For $a,b\in T$, there exists $k\in\BN$ such that
        \begin{equation*}
            (x_1-x_2)^k[Y_0(a,x_1),Y_0(b,x_2)]=0.
        \end{equation*}
\item $V$ is linearly spanned by  $a^{(1)}_{n_1}\cdots a^{(r)}_{n_r}\1$
       for $r\ge 0,\ a^{(1)},\ldots,a^{(r)}\in T,\  n_1,\ldots,n_r\in \BZ$.
\item For $a\in T$,
        \begin{equation*}
            e^{z\D}Y_0(a,x)e^{-z\D}=e^{z\partial_x}Y_0(a,x).
        \end{equation*}
       \end{enumerate}
Then $Y_0$ can be extended uniquely to a linear map $Y:V\to \Hom(V,V((x)))$
such that $(V,Y,\1)$ is a vertex algebra.
\end{theorem}

\begin{remark}  We here comment on Theorem \ref{th:571} and its proof.
Note that in the case of characteristic zero, in the place of (iv) is the D-bracket-derivative formula:
$[\D,Y_0(a,x)]=\frac{d}{dx}Y_0(a,x)$ for $a\in T$. One proof  (see \cite{LL} for example)
uses a conceptual result of \cite{Li96} and a variation of  Proposition \ref{pvacuum-like-isomorphism}.
 The conceptual result is that for any vector space $W$,
every local subset $U$ of $\Hom (W,W((x)))$ generates a vertex algebra $\langle U\rangle$ and
$W$ is a faithful module for $\langle U\rangle$ with $Y_{W}(\alpha(x),x_0)=\alpha(x_0)$
for $\alpha(x)\in \langle U\rangle$. It is straightforward to check that this theorem holds for any field.
As with Proposition \ref{pvacuum-like-isomorphism} (and Lemma \ref{vacuum-like}), note that
 if ${\rm char}\ \BF=0$, one has $\D^{(n)}=\frac{1}{n!}\D^{n}$ for $n\ge 0$ with $\D=\D^{(1)}$
(cf. \cite{LL}). In this case,  the D-bracket-derivative formula is equivalent to the conjugation formula.
In case ${\rm char}\ \BF>0$, we need the conjugation formula, which is stronger
than the D-bracket-derivative formula.
\end{remark}

We end up this section with the following lemma which follows from the same argument
as in the case of characteristic zero (see \cite{Li96}, \cite{LL}):

\begin{lemma}
Let $V$ be a vertex algebra, let $u,v,w^{(0)},\ldots,w^{(k)}\in V$,
and let $(W,Y_W)$ be a faithful $V$-module. Then
\begin{equation}\label{eq:567-1}
    [Y(u,x_1),Y(v,x_2)]=\sum_{i=0}^k Y(w^{(i)},x_2)\partial_{x_2}^{(i)}x_1^{-1}\del{x_2}{x_1}
\end{equation}
on $V$ if and only if the analogous relation holds on $W$:
\begin{equation}\label{eq:567-2}
    [Y_W(u,x_1),Y_W(v,x_2)]=\sum_{i=0}^k Y_W(w^{(i)},x_2)\partial_{x_2}^{(i)}x_1^{-1}\del{x_2}{x_1}.
\end{equation}
In this case, we have
\begin{eqnarray}\label{eq:567-3}
u_iv=w^{(i)}\quad\text{for }0\le i\le k \  \mbox{ and }\  u_iv=0\quad\text{for }i>k.
\end{eqnarray}
On the other hand,
\eqref{eq:567-1} implies \eqref{eq:567-2} and \eqref{eq:567-3} regardless whether $W$ is faithful or not.
\end{lemma}


\section{Vertex algebras associated to general affine Lie algebras}
In this section, we study vertex algebras associated to general affine
Lie algebras and then we present some basic results.
Here and after, we assume $\BF$ has prime characteristic $p$.

Let $H$ be a bialgebra.
Recall that a Lie algebra $\fg$ is called an {\em $H$-module Lie algebra} if
$\fg$ is an $H$-module such that
\begin{equation*}
    h\cdot[a,b]=\sum [h_{(1)}\cdot a,h_{(2)}\cdot b]
\end{equation*}
for $h\in H,\ a,b\in \fg$, where $\Delta(h)=\sum h_{(1)}\otimes h_{(2)}$.
A unital algebra $A$ is called an {\em $H$-module algebra} if
$A$ is an $H$-module such that
\begin{equation*}
    h\cdot (ab)=\sum (h_{(1)}\cdot a) (h_{(2)}\cdot b),\  \  \quad h\cdot 1=\varepsilon(h)1
\end{equation*}
for $h\in H,\ a,b\in A$.

\begin{remark}\label{rLaurent-poly}
 We have seen (see (\ref{eq:partialm})) that
 the Laurent polynomial ring $\BF[t,t^{-1}]$ is a $\B$-module with $\D^{(n)}$ acting as
$\partial_{t}^{(n)}$ for $n\in \BN$, where
$$\partial_{t}^{(n)}t^{m}=\binom{m}{n}t^{m-n}\  \   \   \mbox{ for }m\in \BZ.$$
 Furthermore, it is straightforward to show that $\BF[t,t^{-1}]$ is a $\B$-module algebra.
 \end{remark}

Let $H$ be a bialgebra and let $A$ be an $H$-module algebra.
An  {\em $(A,H)$-module} is an $A$-module $W$ which is also an $H$-module such that
$$h(aw)=\sum (h_{(1)}a)(h_{(2)}w)\   \   \   \mbox{ for }h\in H,\ a\in A,\ w\in W.$$
It is clear that the left adjoint $A$-module is automatically an $(A,H)$-module.

The following is a well known fact:

\begin{lemma}\label{th:ModuleAlg}
Let $\fg$ be an $H$-module Lie algebra.
Then the universal enveloping algebra $U(\fg)$ is naturally an $H$-module algebra.
Moreover, in case $H=\B$,   for $u\in U(\fg)$ we have
\begin{equation}
  e^{x\D}ue^{-x\D}=\left(e^{x\D}u\right)
\end{equation}
on any $(U(\fg),\B)$-module, in particular on $U(\fg)$.
\end{lemma}

\begin{proof} First of all, the tensor algebra $T(\fg)$ is naturally an $H$-module algebra,
where
\begin{equation}\label{eq:module-alg-01}
    h(u\otimes v)=\sum h_{(1)}u \otimes h_{(2)}v
\end{equation}
for $h\in H,\  u,v\in T(\fg)$.  Set
$$U={\rm span}\{[a,b]-a\otimes b+b\otimes a  \mid a,b\in \fg\}\in T(\fg).$$
Then let $I$ be the ideal of $T(\fg)$ generated by $U$.
For $h\in H,\ a,b\in \fg$, we have
\begin{eqnarray*}
h([a,b]-a\otimes b+b\otimes a)=\sum \left([h_{(1)}a,h_{(2)}b]-h_{(1)}a\otimes h_{(2)}b+h_{(1)}b\otimes h_{(2)}a\right)\in U.
\end{eqnarray*}
Thus $U$ is an $H$-submodule of $T(\fg)$. Furthermore, for $h\in H,\ X,Y\in T(\fg),\ u\in U$, we have
$$h(XuY)=\sum (h_{(11)}X)(h_{(12)}u)(h_{(2)}Y)\in I,$$
where $\Delta(h_{(1)})=\sum h_{(11)}\otimes h_{(12)}$.
This shows that $I$ is an $H$-submodule of $T(\fg)$.
It follows that $U(\fg)$, which equals $T(\fg)/I$, is an $H$-module algebra.

Assume $H=\B$ and let $W$ be a $(U(\fg),\B)$-module.  For $u\in U(\fg),\ w\in W$, from (\ref{ecoproduct})  we have
$$e^{x\D}(uw)=(e^{x\D}u)(e^{x\D}w).$$
 A slightly different version of this is
 $$e^{x\D}(ue^{-x\D}w)=(e^{x\D}u)w,$$
 which immediately gives the desired result.
%
\end{proof}

Let $\fg$ be a Lie algebra equipped with a symmetric invariant bilinear form
$\langle\cdot,\cdot\rangle$.
The affine Lie algebra $\widehat{\fg}$ associated to the pair $(\fg,\langle\cdot,\cdot\rangle)$ is the Lie algebra
with the underlying vector space
\begin{equation*}
    \widehat{\fg}=\fg\otimes\BC[t,t^{-1}]\oplus\BC \bk,
\end{equation*}
where $\bk$ is
a nonzero central element of $\widehat\fg$ and
\begin{equation*}
    [a\otimes t^m,b\otimes t^n]=[a,b]\otimes t^{m+n}+m\langle a,b\rangle\delta_{m+n,0}\bk
\end{equation*}
for $a,b\in\fg,\  m,n\in\BZ$.

We call a $\widehat\fg$-module $W$ a {\em restricted module} if for every $a\in\fg$ and $w\in W$,
$a(n)w=0$ for $n$ sufficiently large, where $a(n)$ denotes the operator on $W$ corresponding to $a\otimes t^n$.
On the other hand, if $\bk$ acts as a scalar $\ell\in \BF$ on a $\widehat\fg$-module $W$
we say $W$ is of {\em level} $\ell$.

For $a\in\fg$, form a generating function
\begin{equation}
    a(x)=\sum_{n\in\BZ}(a\otimes t^n)x^{-n-1}.
\end{equation}
Then
\begin{equation*}
    [a(x_1),b(x_2)]=[a,b](x_2)x_2^{-1}\del{x_1}{x_2}-\+{a,b}\partial_{x_1}x_2^{-1}\del{x_1}{x_2}\bk.
\end{equation*}
Using \eqref{eq:partial02}, we get
\begin{equation}\label{eq:affine-local}
    (x_1-x_2)^2[a(x_1),b(x_2)]=0.
\end{equation}

Set $$\widehat{\fg}_{+}=\coprod_{n>0}\fg\otimes t^{-n},\   \   \   \
\widehat{\fg}_{-}=\coprod_{n>0}\fg\otimes t^{n},\   \   \  \   \widehat{\fg}_{(0)}=\fg\oplus \BF\bk.$$
Let $\ell\in\BF$. Let $\widehat{\fg}_{-}$ and $\fg$ act trivially on $\BF$ and let $\bk$
act as scalar $\ell$, making $\BF$ a $\widehat{\fg}_{-}\oplus\widehat{\fg}_{(0)}$-module,
which we denote by $\BF_\ell$.
Form an induced module
\begin{equation}
    V_{\widehat{\fg}}(\ell,0)=U(\widehat{\fg})\otimes_{U(\widehat{\fg}_{-}\oplus\widehat{\fg}_{(0)})}\BF_\ell.
\end{equation}

Define an action of $\B$ on $\widehat\fg$ by
\begin{eqnarray*}
   \D^{(n)}\cdot \bk=\delta_{n,0}\bk,\    \   \   \   \   \
   \D^{(n)}\cdot (a\otimes t^m)=(-1)^n \binom{m}{n}(a\otimes t^{m-n})
   \end{eqnarray*}
 for $n\in \BN,\  a\in\fg,\ m\in\BZ.$
 In terms of generating functions, we have
\begin{eqnarray}
e^{z\D} \cdot \bk= \bk\ \left(=\varepsilon (e^{z\D})\bk\right),\  \  \   \   \  e^{z\D}(a(x))=e^{z\partial_{x}}(a(x))=a(x+z).
\end{eqnarray}
For $n\in\BN$, $a,b\in \fg,\ r,s\in \BZ$, we have
\begin{align*}\label{eq:module-lie-algebra-0}
   &\ \quad \D^{(n)}([a\otimes t^r,b\otimes t^s])\\
   &=[a,b]\otimes (-1)^{n}\binom{r+s}{n}t^{r+s-n}+r\langle a,b\rangle \delta_{r+s,0}\delta_{n,0}\bk\\
   &=\sum_{i=0}^n [a,b]\otimes (-1)^{n}\binom{r}{n-i}\binom{s}{i}t^{r+s-n}+(r-n+i)\langle a,b\rangle \delta_{r+s-n,0}\delta_{n,0}\bk\\
   &=\sum_{i=0}^n [\D^{(n-i)}(a\otimes t^r),\D^{(i)}(b\otimes t^s)].
\end{align*}
It follows that $\widehat\fg$ is a $\B$-module Lie algebra.
In view of Lemma~\ref{th:ModuleAlg}, $U(\widehat\fg)$ is a $\B$-module algebra.

Clearly $\widehat\fg_{-}+\fg+ \BF(\bk-\ell)$ is a $\B$-submodule of $U(\widehat\fg)$,
so that $U(\widehat\fg)(\widehat\fg_{-}+ \fg+ \BF(\bk-\ell))$ is a $\B$-submodule.
Since
\begin{equation*}
    V_{\widehat\fg}(\ell,0)=U(\widehat\fg)/ U(\widehat\fg)(\widehat\fg_{-}+ \fg+ \BF(\bk-\ell))
\end{equation*}
as a $U(\widehat\fg)$-module, it follows that $V_{\widehat\fg}(\ell,0)$ is a $(U(\widehat\fg),\B)$-module.
Furthermore, by Lemma~\ref{th:ModuleAlg}, we have
\begin{equation}\label{eq:affconj}
    e^{z\D}a(x)e^{-z\D}=e^{z\D}(a(x))=a(x+z)
\end{equation}
on $V_{\widehat\fg}(\ell,0)$ for $a\in \fg$.
With \eqref{eq:affine-local} and \eqref{eq:affconj}, by Theorem~\ref{th:571}
we immediately have:

\begin{proposition}\label{th:affineVA}
There exists a vertex algebra structure
on $V_{\widehat\fg}(\ell,0)$, which is uniquely determined by the condition that $\1=1\otimes 1$ is the vacuum vector and
\begin{equation}
    Y(a,x)=a(x)\in(\End V_{\widehat\fg}(\ell,0))[[x,x^{-1}]]\quad \text{for }a\in\fg.
\end{equation}
\end{proposition}

Just as in the case of characteristic zero, we have (cf. \cite{LL}):

\begin{proposition}
Every $V_{\widehat\fg}(\ell,0)$-module $W$
is naturally a restricted $\widehat\fg$-module of level $\ell$
with $a(x)=Y_W(a,x)$ for $a\in\fg$.
On the other hand, on any restricted $\widehat\fg$-module $W$ of level $\ell$, there exists
a $V_{\widehat\fg}(\ell,0)$-module structure $Y_{W}(\cdot,x)$ which is uniquely determined by
$$Y_{W}(a,x)=a(x)\   \   \  \mbox{ for }a\in \fg \subset V_{\widehat\fg}(\ell,0).$$
\end{proposition}

\section{Representations of certain Heisenberg algebras}

For the purpose to study Heisenberg vertex algebras and their representations,
in this section we study representations of certain Heisenberg algebras and we establish a complete reducibility
theorem for a certain category of highest weight type modules.

Let $S$ be a nonempty set. To $S$, we associate a Heisenberg Lie algebra
\begin{eqnarray}
\mathcal{H}_{S}=\sum_{\alpha\in S}\left(\BF a_{\alpha}\oplus \BF b_{\alpha}\right)\oplus \BF {\bf k}
\end{eqnarray}
with a designated basis $\{ \bk\}\cup \{a_{\alpha},\  b_{\alpha} \ |\ \alpha \in S\}$, where ${\bf k}$ is central and
\begin{eqnarray}
[a_{\alpha},a_{\beta}]=[b_{\alpha},b_{\beta}]=0,\   \   \   \  [a_{\alpha},b_{\beta}]=\delta_{\alpha,\beta}{\bf k}
\   \   \   \   \mbox{ for }\alpha,\beta\in S.
\end{eqnarray}
(Recall that a Heisenberg algebra is a Lie algebra whose derived subalgebra coincides
with its center which is assumed to be $1$-dimensional.)
Set
\begin{eqnarray}
P[S]=\BF[x_{\alpha}\mid \alpha\in S],
\end{eqnarray}
the commutative polynomial algebra.
It is well known that for any $\ell\in \BF$,  $P[S]$ becomes an $\mathcal{H}_{S}$-module where
${\bf k}$ acts as scalar $\ell$ and  for $\alpha\in S$, $a_{\alpha}$ acts as the differential operator $\ell \partial_{x_{\alpha}}$ while $b_{\alpha}$ acts as the multiplication operator $x_{\alpha}$. Let $J_{\ell}$ be the left ideal of $U(\mathcal{H}_{S})$
generated by ${\bf k}-\ell$ and $a_{\alpha}$ for $\alpha\in S$.
It can be readily seen that
\begin{eqnarray}
P[S]\simeq U(\mathcal{H}_{S})/J_{\ell}
\end{eqnarray}
as an $\mathcal{H}_{S}$-module. Set
\begin{eqnarray}
A(S,\ell)=U(\mathcal{H}_{S})/({\bf k}-\ell)U(\mathcal{H}_{S}),
\end{eqnarray}
a commutative associative algebra.

Furthermore, let $f: S\rightarrow \BF$ be any function. Denote by $I_{f}$ the ideal of $P[S]$ generated by elements
$x_{\alpha}^{p}-f(\alpha)$ for $\alpha\in S$. Notice that
$$\partial_{x_{\beta}}(x_{\alpha}^{p}-f(\alpha))=\delta_{\alpha,\beta}px_{\alpha}^{p-1}=0$$
for $\alpha,\beta\in S$.
It follows that $I_{f}$ is an $\mathcal{H}_{S}$-submodule.
Then we set
\begin{eqnarray}
V[S,f]=\BF[x_{\alpha}\  |  \  \alpha\in S]/I_{f},
\end{eqnarray}
a commutative associative algebra and an $\mathcal{H}_{S}$-module of level $\ell$.
It can be readily seen that the commutative products
\begin{eqnarray}\label{ebasis-0}
x_{\alpha_1}^{k_1}x_{\alpha_2}^{k_2}\cdots x_{\alpha_r}^{k_r}
\end{eqnarray}
for $r\ge 0,\  \alpha_1,\dots, \alpha_r\in S$ (distinct), $0\le k_1,\dots, k_{r}\le p-1$ form a basis
of $V[S,f]$.

\begin{lemma}\label{heisenberg-S}
For any nonzero $\ell\in \BF$, the $\mathcal{H}_{S}$-module $V[S,f]$ of level $\ell$ is irreducible.
\end{lemma}

\begin{proof} First, we consider the special case with $|S|=1$. In this case, $\mathcal{H}_{S}=\BF a+\BF b+\BF {\bf k}$
and $a$ acts as $\ell\frac{d}{dx}$ and $b$ acts as the multiplication operator $x$ on
the polynomial algebra $\BF[x]$. Let $\alpha\in \BF$.  Set $V[\alpha]=\BF[x]/(x^{p}-\alpha)\BF[x]$, which is $p$-dimensional.
Notice that with $\ell\ne 0$, any submodule of $V[\alpha]$ is $(x\frac{d}{dx})$-stable and $(x\frac{d}{dx})x^{n}=nx^{n}$ for $n\in \BN$. Then it follows that $V[\alpha]$ is an irreducible module.
Furthermore, let $w$ be a nonzero vector in an $\mathcal{H}_{S}$-module such that $aw=0$ and $b^{p}w=\alpha w$.
By using $\BF[x]$ and the irreducibility of $V[\alpha]$, it is straightforward to show that $V[\alpha]\simeq U(\mathcal{H}_{S})w$.

Next, we consider the case with $S$ finite.
For $\alpha\in S$, let $A_{\alpha}$ be the subalgebra of $A(S,\ell)$ generated by
$a_{\alpha}$ and $b_{\alpha}$. Then
$$A(S,\ell)\simeq \bigotimes_{\alpha\in S}A_{\alpha}.$$
It can be readily seen that $V[S,f]=\otimes_{\alpha\in S}V[f_\alpha]$. For each $\alpha\in S$,
it was proved previously that  $V[f_\alpha]$ is a $p$-dimensional irreducible $A_\alpha$-module.
Then it follows that $V[S,f]$ is an irreducible $\mathcal{H}_{S}$-module.

Now, we consider the general case. To show that $V[S,f]$ is an irreducible $\mathcal{H}_{S}$-module, we prove that
 $U(\mathcal{H}_{S})w=V[S,f]$ for any nonzero vector $w\in V[S,f]$.
Let $0\ne w\in V[S,f]$. There exists a finite subset $S'$ of $S$ such that $w\in U(\mathcal{H}_{S'} )1$.
We have already proved that $U(\mathcal{H}_{S'} )1$ is an irreducible $\mathcal{H}_{S'}$-module, which implies
$U(\mathcal{H}_{S'} )1=U(\mathcal{H}_{S'} )w$.
Then $$1\in U(\mathcal{H}_{S'} )1=U(\mathcal{H}_{S'} )w\subset U(\mathcal{H}_{S})w.$$
It follows that $V[S,f]=U(\mathcal{H}_{S})w.$ Therefore, $V[S,f]$ is an irreducible $\mathcal{H}_{S}$-module.
\end{proof}

Let $d$ be a positive integer. Set
\begin{eqnarray}
S_{d}=\{ (i,n)\ |\  1\le i\le d,\ n\in \BZ_+\setminus p\BZ_+\}.
\end{eqnarray}

\begin{definition}
Let $\lambda: S_{d}\rightarrow \BF$ be a function.
Define $P[S_{d},\lambda]$ to be the unital commutative associative algebra generated
by $x_{i,n}$ for $(i,n)\in S_{d}$, subject to relations
$$x_{i,n}^{p}=\lambda_{i,n}^{p}\   \   \   \mbox{ for }(i,n)\in S_{d}.$$
\end{definition}

It follows that the commutative products
\begin{eqnarray}\label{ebasis}
x_{i_1,n_1}^{k_1}x_{i_2,n_2}^{k_2}\cdots x_{i_r,n_r}^{k_r}
\end{eqnarray}
for $r\ge 0,\ (i_1,n_1),\dots,(i_r,n_r)\in S_{d}$ (distinct), $0\le k_{1},\dots,k_r\le p-1$ form a basis
of $P[S_{d},\lambda]$.

Let $\fh$ be a finite-dimensional vector space over $\BF$ equipped with a
non-degenerate symmetric bilinear form $\langle\cdot,\cdot\rangle$.
View $\fh$ as an abelian Lie algebra, so that  $\langle\cdot,\cdot\rangle$ is
an invariant bilinear form. Then we have an affine Lie algebra $\widehat\fh$ associated to the pair
$(\fh,\langle\cdot,\cdot\rangle)$, where
$$[\alpha(m),\beta(n)]=m\delta_{m+n,0}\langle \alpha,\beta\rangle \bk$$
for $\alpha,\beta\in \fh,\ m,n\in \BZ$.

We have the following simple result:

\begin{lemma}\label{lcentral}
The elements $u(kp)$ and $u(k)^{p}$  for $u\in \fh,\ k\in \BZ$ lie in the center
of $U(\widehat{\fh})$.
\end{lemma}

\begin{proof} Let $u\in \fh,\ k\in \BZ$. For any $v\in \fh,\ n\in \BZ$, we have
$$[u(kp),v(n)]=(kp)\delta_{kp+n,0}\langle u,v\rangle {\bf k}=0.$$
It follows that $u(kp)$ is central in $U(\widehat{\fh})$.
On the other hand, for any $u,v\in \fh,\ m,n\in \BZ$, by induction on $r$, we have
$$[u(m)^{r},v(n)]=rm\delta_{m+n,0}\langle u,v\rangle u(m)^{r-1}{\bf k}$$
for all $r\ge 1$, which implies
$$[u(m)^{p},v(n)]=pm\delta_{m+n,0}\langle u,v\rangle u(m)^{p-1}{\bf k}=0.$$
This shows that $u(k)^{p}$ for $u\in \fh,\ k\in \BZ$ are central in $U(\widehat{\fh})$.
\end{proof}

Recall that $\widehat{\fh}_+=\coprod_{n>0}\fh\otimes t^{-n}$ and $\widehat{\fh}_-=\coprod_{n>0}\fh\otimes t^{n}$.
For any $\widehat\fh$-module $W$, set
\begin{equation}
    \Omega_W=\{w \in W\mid \widehat\fh_- w=0\}.
\end{equation}
A nonzero element of $\Omega_W$ is called a {\em vacuum vector}. It is clear that
every central element of $U(\widehat{\fh})$ preserves $\Omega_W$. In particular, $u(kp)$ and $u(k)^{p}$ preserve
$\Omega_{W}$ for any $u\in \fh,\ k\in \BZ$.

\begin{definition}\label{def:C}
We say that a restricted $\widehat\fh$-module $W$
satisfies condition $\mathcal{C}_{0}$ if
\begin{enumerate}[(i)]
\item $u(np)$ and $u(n)^p$ act trivially on $W$ for $u\in \fh,\  n\in \BZ_+$.\label{con:3}
\item $u(-np)$ and $u(-n)^{p}$ act on $W$ semisimply for $u\in \fh,\ n\in \BN$.
\end{enumerate}
\end{definition}

\begin{lemma}\label{th:vacuum}
Let $W$ be any nonzero restricted $\widehat\fh$-module satisfying condition $\mathcal{C}_0$. Then
$\Omega_{W}\ne 0$.
\end{lemma}

\begin{proof} Let $w$ be a nonzero vector in $W$. If $\fh (n)w=0$ for all $n\ge 1$, then $w\in \Omega_{W}$, so that
$\Omega_{W}\ne 0$. Now, assume $\fh (n)w\ne 0$ for some $n\ge 1$.
Since $W$ is restricted and since $\fh$ is finite dimensional,  there exists a positive integer $k$
such that $\fh(k)w\ne 0$ and $\fh (n)w=0$ for all $n>k$. Set
$$\mathcal{L}_{k}=\fh(1)+\cdots +\fh(k),$$
a finite-dimensional (abelian) subalgebra of $\widehat{\fh}_{-}$.
Note that $\mathcal{L}_{k}w\ne 0$  and that by Definition~\ref{def:C} \eqref{con:3},
for $u\in \fh,\ n\in\BZ_+$, $u(n)$  acts nilpotently on $W$. Then there exists a nonnegative integer $k_0$
such that $(\mathcal{L}_{k})^{k_0}w\ne 0$ and $(\mathcal{L}_{k})^{k_0+1}w=0$.
Consequently, we have $(\widehat\fh_-)^{k_0+1} w=0$
and $(\widehat\fh_-)^{k_0} w\ne0$.
Then we have $(\widehat\fh_-)^{k_0} w\subset \Omega_{W}$, proving $\Omega_{W}\ne 0$.
\end{proof}

Let $\ell\in \BF$ and let $\lambda_{0}\in\fh^*$.
Denote by $\BF_{\ell,\lambda_0}$ the one-dimensional $(\widehat\fh_-\oplus \fh\oplus \BF {\bf k})$-module $\BF$
with $\widehat\fh_{-}$
acting trivially, with $h$ acting as scalar $\lambda_0(h)$ for $h\in \fh$, and with $\bk$ acting as scalar $\ell$.
Form an induced module
\begin{equation}
    M(\ell,\lambda_0)
    =U(\widehat\fh)\otimes_{U(\widehat\fh_-\oplus\fh\oplus \BF {\bf k})}\BF_{\ell,\lambda_0}
    \simeq U(\widehat\fh_{+}),
\end{equation}
which is a restricted $\widehat\fh$-module of level $\ell$.
Set
$$\1_{\ell,\lambda_0}=1\otimes 1\in M(\ell,\lambda_0).$$

Furthermore, let $\lambda\in (\widehat\fh_+)^*$. For $u\in \fh,\ n\in \BZ_+$,
alternatively write $\lambda_{n}(u)=\lambda(u(-n))$.
Denote by $J(\lambda)$ the submodule of $M(\ell,\lambda_0)$ generated by vectors
\begin{eqnarray}\label{especial-vacuum}
(u(-m)-\lambda_{m}(u))\1_{\ell,\lambda_0},\  \  \  \   (u(-n)^p-\lambda_{n}(u)^p)\1_{\ell,\lambda_0}
\end{eqnarray}
 for $u\in \fh$, $m\in p\BZ_+$, $n\in \BZ_+\setminus p\BZ_+$. Note that  for $u\in \fh,\ n\in p\BZ_+$, we also have
 $$(u(-n)^{p}-\lambda_{n}(u)^{p})\1_{\ell,\lambda_0}\in J(\lambda)$$
 as
 $$(u(-n)^{p}-\lambda_{n}(u)^{p})\1_{\ell,\lambda_0}=(u(-n)-\lambda_{n}(u))^{p-1}(u(-n)-\lambda_{n}(u))\1_{\ell,\lambda_0}.$$
 With Lemma \ref{lcentral}, it can be readily seen that those vectors in (\ref{especial-vacuum})
lie in $\Omega_{M(\ell,\lambda_0)}$.
Set
\begin{equation}
    L_{\widehat\fh}(\ell,\lambda_0,\lambda)=M(\ell,\lambda_0)/ J(\lambda).
\end{equation}
 We have:

\begin{proposition}\label{th:L-irr}
Assume that $\BF$ is algebraically closed.
Let $\ell\in \BF$, $\lambda_0\in \fh^{*},\ \lambda\in (\widehat{\fh}_{+})^{*}$ with $\ell \ne 0$.
Then $L_{\widehat\fh}(\ell,\lambda_0,\lambda)$ is an irreducible $\widehat\fh$-module.
Furthermore, if $w$ is a vacuum vector in some $\widehat\fh$-module $W$  of level $\ell$ such that
\begin{enumerate}[(i)]
\item $u(0)w=\lambda_{0}(u)w$ for $u\in \fh$,
\item $u(-n)w=\lambda_{n}(u)w$ for $u\in \fh,\ n\in p\BZ_+$,
\item $u(-n)^pw=\lambda_{n}(u)^pw$ for $u\in \fh,\  n\in \BZ_+\setminus p\BZ_{+}$,
\end{enumerate}
then $U(\widehat{\fh})w$ is isomorphic to $L_{\widehat\fh}(\ell,\lambda_0,\lambda)$.
\end{proposition}

\begin{proof}   For $m\in \BZ$, write $\fh(m)=\fh\otimes t^{m}$. Set
\begin{eqnarray}
\widehat{\fh}''=\coprod_{n\in \BZ\setminus p\BZ}\fh(n) \oplus \BF{\bf k},
\end{eqnarray}
which  is a subalgebra of $\widehat{\fh}$ and  a Heisenberg algebra itself.
As $\BF$ is algebraically closed, there is an orthonormal basis $\{u^{(1)},\ldots,u^{(d)}\}$ of $\fh$.
 Recall
$S_{d}=\{ (i,n)\ |\  1\le i\le d,\ n\in \BZ_+\setminus p\BZ_+\}.$
Associated to $S_{d}$, we have a Heisenberg algebra $\mathcal{H}_{S_{d}}$
and a commutative polynomial algebra
$$P[S_{d}]=\BF[x_{i,n}\ |\  (i,n)\in S_{d}],$$
which is naturally an $\mathcal{H}_{S_{d}}$-module of level $\ell$.
It can be readily seen that $\widehat{\fh}''\simeq \mathcal{H}_{S_{d}}$
with $\frac{1}{n}u^{(i)}(n)=a_{i,n}$ and $u^{(i)}(-n)=b_{i,n}$ for $(i,n)\in S_{d}$.
(Note that $n\ne 0$ in $\BF$ with $(i,n)\in S_d$.)
Then  $P[S_{d}]$ becomes an $\widehat{\fh}''$-module with
$${\bf k}=\ell,\   \   \   u^{(i)}(n)=\ell n\partial_{x_{i,n}}, \   \   \  u^{(i)}(-n)=x_{i,n}
\   \   \mbox{ for }(i,n)\in S_{d}.$$
For $1\le i\le d,\  n\in\BZ_+$, set
\begin{eqnarray}
\lambda_{i,n}=\lambda (u^{(i)}(-n)).
\end{eqnarray}
Note that $\lambda$ naturally gives rise to a function from $S_{d}$ to $\BF$, denoted by $\lambda$ again.
We then have a commutative associative algebra
\begin{eqnarray}
P[S_{d},\lambda]=V[S_{d},\lambda]=P[S_{d}]/I_{\lambda},
\end{eqnarray}
which is also an $\widehat{\fh}''$-module of level $\ell$, where
$$I_{\lambda}=\sum_{(i,n)\in S_{d}}(x_{i,n}^{p}-\lambda_{i,n}^{p})P[S_{d}],$$
an ideal  and an $\widehat{\fh}''$-submodule of $P[S_{d}]$.
With $\ell\ne 0$, from Lemma \ref{heisenberg-S},
$P[S_d,\lambda]$ is an irreducible $\widehat{\fh}''$-module.

For $u\in \fh,\ k\in p\BZ_+$, let $u(-k)$ act on $P[S_d]$ as scalar
$\lambda_{k}(u)$, let $u(k)$ act trivially, and let $u(0)$ act as scalar $\lambda_{0}(u)$. Then
$P[S_d]$ becomes an $\widehat{\fh}$-module of level $\ell$.
Furthermore, $P[S_d,\lambda]$ becomes an irreducible $\widehat{\fh}$-module of level $\ell$.

From the construction of $M(\ell,\lambda_0)$, there exists an $\widehat{\fh}$-module homomorphism
 $\theta: M(\ell,\lambda_0)\rightarrow  P[S_d,\lambda]$
with $\theta({\bf 1}_{\ell,\lambda_0})=1$. It can be readily seen that $\theta$ reduces to a homomorphism $\bar{\theta}$ from
 $L_{\widehat{\fh}}(\ell,\lambda_0,\lambda)$ onto   $P[S_d,\lambda]$. Note that
 $$L_{\widehat{\fh}}(\ell,\lambda_0,\lambda)=U(\widehat{\fh}){\bf 1}_{\ell,\lambda_0}
 =U(\widehat{\fh}''_{+}){\bf 1}_{\ell,\lambda_0}
 =S(\widehat{\fh}''_{+}){\bf 1}_{\ell,\lambda_0}=\left(S(\widehat{\fh}''_{+})/J(\lambda)''\right){\bf 1}_{\ell,\lambda_0},$$
where $J(\lambda)''$ denotes the ideal of $S(\widehat{\fh}''_{+})$, generated by
$u^{(i)}(-n)^{p}-\lambda_{i,n}^{p}$ for $(i,n)\in S_d$. We have
$$S(\widehat{\fh}''_{+})/J(\lambda)''\simeq P[S_d,\lambda]$$
as an algebra with $u^{(i)}(-n)+J(\lambda)''$ corresponding to $x_{i,n}$ for $(i,n)\in S_{d}$.
It follows that $\bar{\theta}$ is isomorphism.
 Consequently, $L_{\widehat{\fh}}(\ell,\lambda_0,\lambda)$ is an irreducible
 $\widehat{\fh}$-module. The second assertion follows immediately from the construction and irreducibility of
 $L_{\widehat{\fh}}(\ell,\lambda_0,\lambda)$.
\end{proof}

Recall that  with $d=\dim \fh$ and $S_{d}=\{ (i,n)\ |\  1\le i\le d,\ n\in \BZ_+\setminus p\BZ_+\}$
and recall  the unital commutative associative algebra $P[S_{d},\lambda]$.
For $(i,n)\in S_{d}$, set
\begin{eqnarray}
P[S_{d},\lambda]_{i,n}^{o}=\langle  x_{j,m}\ |\  (j,m)\in S_{d}\setminus \{ (i,n)\}  \rangle
\end{eqnarray}
(the subalgebra generated by $ x_{j,m}$ for $(j,m)\in S_{d}$ with $(j,m)\ne (i,n)$). Note that
\begin{eqnarray}
P[S_{d},\lambda]=P[S_{d},\lambda]_{i,n}^{o}\oplus x_{i,n}P[S_{d},\lambda]_{i,n}^{o}\oplus \cdots \oplus x_{i,n}^{p-1}P[S_{d},\lambda]_{i,n}^{o}.
\end{eqnarray}
Furthermore, we set
\begin{eqnarray}
P[S_{d},\lambda]_{i,n}'=\partial_{x_{i,n}}P[S_{d},\lambda]
=P[S_{d},\lambda]_{i,n}^{o}+x_{i,n}P[S_{d},\lambda]_{i,n}^{o}+\cdots +x_{i,n}^{p-2}P[S_{d},\lambda]_{i,n}^{o},
\end{eqnarray}
a subspace of $P[S_{d},\lambda]$.

For any subset $T$ of $S_{d}$, set
\begin{eqnarray}
P[S_{d},\lambda]_{T}=\langle  x_{i,n}\ |\  (i,n)\in T  \rangle.
\end{eqnarray}
We have:

\begin{lemma}\label{th:existenceS}
Let $T$ be a subset of $S_{d}$, $T_0$ a finite subset of $T$, and
let $\{f_{i,n}\ |\  (i,n)\in T_0\}$ be a set of elements of $P[S_{d},\lambda]_{T}$ with $f_{i,n}\in P[S_{d},\lambda]_{i,n}'$ such that
\begin{equation}\label{eq:existenceS-01}
    n\partial_{x_{i,n}}f_{j,m}=m\partial_{x_{j,m}}f_{i,n}\   \    \   \mbox{ for  } (i,n),\ (j,m)\in T_0.
\end{equation}
Then there exists $f\in P[S_{d},\lambda]_{T}$ such that
\begin{eqnarray*}
   n\partial_{x_{i,n}}f=f_{i,n}\   \    \   \mbox{ for all } (i,n)\in T_0.
\end{eqnarray*}
\end{lemma}

\begin{proof} For $(i,n)\in S_{d}$, define a linear operator $g_{i,n}$ on $P[S_{d},\lambda]$ by
\begin{equation*}
    g_{i,n} (x_{i,n}^k f)=\begin{cases}
        \frac{1}{n(k+1)} x_{i,n}^{k+1}f &\text{if } 0\le k\le p-2;\\
        0&\text{if } k=p-1,
    \end{cases}
\end{equation*}
where $f\in P[S_{d},\lambda]_{i,n}^{o}$.
Then
\begin{equation}\label{eq:existenceS-02}
    m\partial_{x_{j,m}}g_{i,n}=g_{i,n} m\partial_{x_{j,m}}\\
\end{equation}
on $P[S_{d},\lambda]$ for distinct $(i,n),(j,m)\in S_{d}$,
and
\begin{equation}\label{eq:existenceS-03}
    n\partial_{x_{i,n}}g_{i,n} f=f\   \   \   \mbox{ for any }f\in P[S_{d},\lambda]_{i,n}'=\partial_{x_{i,n}}P[S_{d},\lambda].
\end{equation}
In particular,  we have
\begin{equation}\label{eq:existenceS-04}
    n\partial_{x_{i,n}}g_{i,n} f_{i,n}=f_{i,n}\   \    \   \mbox{ for all } (i,n)\in T_0.
\end{equation}
From definition, we have $g_{i,n}P[S_{d},\lambda]_{T}\subset P[S_{d},\lambda]_{T}$
for $(i,n)\in T$.

We now proceed to prove the lemma by induction on $|T_0|$.
For $T_0=\{(i,n)\}$, taking $f=g_{i,n}f_{i,n}$,  by \eqref{eq:existenceS-04} we have
\begin{equation*}
    n\partial_{x_{i,n}}f=n\partial_{x_{i,n}}g_{i,n}f_{i,n}=f_{i,n}.
\end{equation*}
For the induction step, pick up $(j,m)\in T_0$ and set $T'_0=T_0\setminus\{(j,m)\}$.
There exists  $f'\in P[S_{d},\lambda]_{T}$ such that $n\partial_{x_{i,n}}f'=f_{i,n}$
for all $(i,n)\in T_0'$.
Set
\begin{equation*}
    f=f'+g_{j,m}f_{j,m}-g_{j,m}m\partial_{j,m}f'\in P[S_{d},\lambda]_{T}.
\end{equation*}
For $(i,n) \in T_0'$,
noting that $m\partial_{x_{j,m}}$ and $n\partial_{x_{i,n}}$ commute,
using \eqref{eq:existenceS-02} and \eqref{eq:existenceS-01},
we have
\begin{align*}
    n\partial_{x_{i,n}}f&=n\partial_{x_{i,n}}f'+n\partial_{x_{i,n}}g_{j,m}\left(f_{j,m}-m\partial_{x_{j,m}}f'\right)\\
    &=f_{i,n}+g_{j,m}\left(n\partial_{x_{i,n}}f_{j,m}-m\partial_{x_{j,m}}n\partial_{x_{i,n}}f'\right)\\
        &=f_{i,n}+g_{j,m}\left(n\partial_{x_{i,n}}f_{j,m}-m\partial_{x_{j,m}}f_{i,n}\right)\\
    &=f_{i,n}.
\end{align*}
On the other hand, using \eqref{eq:existenceS-04} and \eqref{eq:existenceS-03} we have
\begin{align*}
    m\partial_{x_{j,m}}f&=m\partial_{x_{j,m}}f'+m\partial_{x_{j,m}}g_{j,m}f_{j,m}-m\partial_{x_{j,m}}g_{j,m}(m\partial_{x_{j,m}}f')\\
    &=m\partial_{x_{j,m}}f'+f_{j,m}-m\partial_{x_{j,m}}f'\\
    &=f_{j,m},
\end{align*}
noticing that $m\partial_{x_{j,m}}f'\in P[S_{d},\lambda]_{j,m}'$.
This completes the proof.
\end{proof}

Now we have the following  analog of \cite[Theorem~1.7.3]{FLM}:

\begin{theorem}\label{th:h-mod-reducibility}
Let $\ell \in \BF^{\times}$. Then every restricted $\widehat\fh$-module of  level $\ell$
satisfying condition $\mathcal{C}_0$ is completely reducible.
\end{theorem}

\begin{proof} It is a slight modification of the proof in  \cite{FLM}.
Let $W$ be any restricted $\widehat\fh$-module of level $\ell$ satisfying condition $\mathcal{C}_0$.
For $\lambda_0\in \fh^{*},\ \lambda\in (\widehat{\fh}_{+})^{*}$, set
$$W_{\lambda_0,\lambda}=\{ w\in W\  \mid \  u(-m)w=\lambda_{m}(u)w, \  \  u(-n)^{p}w=\lambda_{n}(u)^pw
\  \  \mbox{ for }u\in \fh,\ m\in p\BN,\ n\in  \BZ_+\}.$$
As $u(-m),\ u(-n)^{p}$ (with $u,m,n$ given as above) lie in the center of $U(\widehat\fh)$, $W_{\lambda_0,\lambda}$
is an $\widehat\fh$-submodule. Furthermore, from  condition $\mathcal{C}_0$, $W$ is a direct sum of submodules $W_{\lambda_0,\lambda}$.
Then it suffices to show that each submodule $W_{\lambda_0,\lambda}$ is completely reducible.

Now, we simply assume $W=W_{\lambda_0,\lambda}$ for some $\lambda_0\in \fh^*,\  \lambda\in (\widehat{\fh}_{+})^*$.
Set $W^{0}=U(\widehat{\fh})\Omega_{W}\subset W$. In view of Proposition \ref{th:L-irr},
$W^{0}$ is a direct sum of irreducible submodules isomorphic to $L_{\widehat{\fh}}(\ell,\lambda_0,\lambda)$.
Then it suffices to show that $W=W^0$. Suppose instead that $W\ne W^{0}$.
As $W/ W^{0}$ is an $\widehat\fh$-module satisfying
condition $\mathcal{C}_0$, it contains a vacuum vector by Lemma \ref{th:vacuum}. That is,
there exists $v\in  W$ such that
$v\notin W^{0}$ and $u^{(i)}(n) v\in W^{0}$ for all $1\le i\le d$ and $n\in\BZ_+$.
We are going to show that there exits  $h\in W^{0}$ such that $u^{(i)}(n)h=u^{(i)}(n)v$ for all
$i=1,\ldots,d$ and $n\in\BZ_+$, so that $v-h\in \Omega_{W}$. Then
$v\in h+\Omega_{W}\subset W^{0}$, a contradiction.
Since $u^{(i)}(n)$ acts trivially on $W$ for $n\in p\BZ_+$ and $1\le i\le d$,
it suffices to find some $h\in W^{0}$ such that $u^{(i)}(n)h=u^{(i)}(n)v$
for $i=1,\ldots,d$ and $n\in\BZ_+\setminus p\BZ_+$.

Choose a basis $\{w_\gamma\}_{\gamma\in \Gamma}$ ($\Gamma$ an index set) of $\Omega_{W}$.
We have
\begin{equation*}
   W^{0}=\bigoplus_{\gamma\in\Gamma} U(\widehat{\fh}) w_\gamma,
\end{equation*}
where $U(\widehat{\fh}) w_\gamma\simeq P[S_{d},\lambda]$.
As $W$ is restricted, there exists $n_0\in\BZ_+$ such that $u^{(i)}(n) v=0$ for $i=1,\dots,d,\ n>n_0$.
Then there is a finite subset $\Gamma_0\subset \Gamma$ such that
$$u^{(i)}(n)v\in \bigoplus_{\gamma\in\Gamma_0} U(\widehat{\fh}) w_\gamma
\   \   \   \mbox{ for all }1\le i\le d,\ n\in \BZ_+.$$
For $(i,n)\in S_{d}$, $\gamma\in\Gamma_0$,
let $s_{in\gamma}$ be the component of $u^{(i)}(n)v$ in
$U(\widehat{\fh}) w_\gamma$ with respect to this decomposition. We have
$s_{in\gamma}=0$ whenever $n>n_0$.
For $(i,n), (j,m)\in S_{d}$, as
\begin{equation*}
    u^{(i)}(n)u^{(j)}(m)v=u^{(j)}(m)u^{(i)}(n)v,
\end{equation*}
we have
$u^{(i)}(n)s_{jm\gamma}=u^{(j)}(m)s_{in\gamma}$ for all $\gamma\in\Gamma_0$.
If we can find $h_\gamma\in U(\widehat{\fh}) w_\gamma$ such that
$u^{(i)}(n)h_\gamma=s_{in\gamma}$ for all $(i,n)\in S_{d}$, $\gamma\in\Gamma_0$,
then we can take $h=\sum_{\gamma\in\Gamma_0}h_\gamma$ and we will be done.

Fix $\gamma\in\Gamma_0$ and identify $U(\widehat{\fh}) w_\gamma$
with $P[S_d,\lambda]$. Then
\begin{equation*}
    n\partial_{x_{i,n}}s_{jm\gamma}=m\partial_{x_{j,m}}s_{in\gamma}\   \   \   \mbox{ for }(i,n), (j,m)\in S_{d}.
\end{equation*}
We claim that $s_{in\gamma}\in P[S_d,\lambda]_{i,n}'$ $\ (=\partial_{x_{i,n}}P[S_d,\lambda])$. Write
$$s_{in\gamma}=a_{0}+a_{1}x_{i,n}+\cdots + a_{p-1}x_{i,n}^{p-1},$$
where $a_{j}\in P[S_d,\lambda]_{i,n}^{o}$.
As $u^{(i)}(n)^{p}W=0$ by assumption, we have $u^{(i)}(n)^{p-1}(u^{(i)}(n)v)=0$.
Consequently,  $u^{(i)}(n)^{p-1}s_{in\gamma}=0$.
Since $u^{(i)}(n)^{p-1} x_{i,n}^{p-1}=(\ell n)^{p-1}(p-1)!$ and $u^{(i)}(n)^{p-1} x_{i,n}^{k}=0$ for $0\le k<p-1$,
we have
$$0=u^{(i)}(n)^{p-1}s_{in\gamma}=(\ell n)^{p-1}(p-1)!a_{p-1},$$
which implies $a_{p-1}=0$. This proves $s_{in\gamma}\in P[S_d,\lambda]_{i,n}'$.
There exists a finite subset $T$ of $S_{d}$ such that $\{ (i,n)\in S_{d}\ |\  n\le n_0\}\subset T$ and
$$s_{in\gamma}\in P[S_d,\lambda]_{T}\  \  \   \mbox{ for all }(i,n)\in S_{d}.$$
Then by Lemma~\ref{th:existenceS} with $T_0=T$, there exists $h\in P[S_d,\lambda]_{T}$ such that
\begin{equation*}
    \ell n\partial_{x_{i,n}}h=s_{in\gamma}\   \   \   \mbox{ for all  }(i,n)\in T.
\end{equation*}
Note that as $h\in P[S_d,\lambda]_{T}$ and $\{ (i,n)\in S_{d}\ |\  n\le n_0\}\subset T$, we have
$$m\partial_{x_{j,m}}h=0=s_{jm\gamma}\   \   \   \mbox{ for }(j,m)\in S_{d}\setminus T.$$
Therefore $\ell n\partial_{x_{i,n}}h=s_{in\gamma}$ for all  $(i,n)\in S_{d}$. This completes the proof.
\end{proof}

\section{ Heisenberg vertex algebras and their modules}

In this section, we study vertex algebras and their modules associated to Heisenberg Lie algebras.
We particularly study their simple quotient vertex algebras and irreducible modules.

Let $\fh$ be a finite-dimensional vector space over $\BF$ equipped with a
non-degenerate symmetric bilinear form $\langle\cdot,\cdot\rangle$.
Viewing $\fh$ as an abelian Lie algebra with $\langle\cdot,\cdot\rangle$
an invariant bilinear form, we have an affine Lie algebra $\widehat\fh$ associated to the pair
$(\fh,\langle\cdot,\cdot\rangle)$.

Let $\ell\in \BF$. We have a vertex algebra
$V_{\widehat\fh}(\ell,0)$, where the underlying space is the $\widehat\fh$-module with generator $\1$, subject to relations
$\bk\cdot {\bf 1}=\ell {\bf 1}$ and $\fh(n)\1=0$ for $n\in \BN$. As a vector space,
\begin{eqnarray}
V_{\widehat\fh}(\ell,0)=U(\widehat{\fh}_{+})=S(\widehat{\fh}_{+}).
\end{eqnarray}
Furthermore, $\fh$ is identified as a subspace of $V_{\widehat\fh}(\ell,0)$ through the linear map
$u\mapsto u(-1)\1$, which generates $V_{\widehat\fh}(\ell,0)$ as a vertex algebra.
Note that every irreducible restricted $\widehat{\fh}$-module of level $\ell$ is an irreducible $V_{\widehat{\fh}}(\ell,0)$-module.
Then $L_{\widehat{\fh}}(\ell,\lambda_0,\lambda)$ for $\lambda_0\in \fh^{*},\ \lambda\in (\widehat{\fh}_{+})^{*}$
are irreducible $V_{\widehat{\fh}}(\ell,0)$-modules.

Unlike in the case of characteristic zero, the vertex algebra $V_{\widehat\fh}(\ell,0)$ is {\em no longer} simple.
In the following, we shall study (determine) simple quotient vertex algebras of $V_{\widehat\fh}(\ell,0)$
and their irreducible modules. First, we have:

\begin{lemma}
For any $u\in \fh,\ n\in \BN$, $u(n)^{p}$ and $u(np)$ act trivially on $V_{\widehat\fh}(\ell,0)$, and
$u(-np){\bf 1}$ and $u(-n)^{p}{\bf 1}$ lie in the center of $V_{\widehat\fh}(\ell,0)$.
\end{lemma}

\begin{proof} Note that $V_{\widehat\fh}(\ell,0)=U(\widehat\fh)\1$. If $a$ is a central element of $U(\widehat\fh)$
such that $a\cdot \1=0$, then $a$ acts trivially on the whole space $V_{\widehat\fh}(\ell,0)$.
Then the first part follows immediately from Lemma \ref{lcentral}. For any $v\in \fh,\ k\in \BN$, as
$v_k=v(k)$ we have
$$v_{k}u(-np){\bf 1}=u(-np)v_{k}{\bf 1}=0\   \mbox{ and }\  v_{k}u(-n)^{p}{\bf 1}=u(-n)^{p}v_{k}{\bf 1}=0,$$
so that
$$[Y(v,x_1),Y(u(-np){\bf 1},x_2)]=0\   \mbox{ and }\ [Y(v,x_1),Y(u(-n)^{p}{\bf 1},x_2)]=0.$$
Since $\fh$ generates $V_{\widehat\fh}(\ell,0)$ as a vertex algebra, it follows that
$u(-np){\bf 1}$ and $u(-n)^{p}{\bf 1}$ lie in the center of $V_{\widehat\fh}(\ell,0)$.
\end{proof}

Recall that every vertex algebra $V$ is naturally a $\B$-module
and a left ideal of $V$ is an ideal if and only if it is a $\B$-submodule
 as $\D^{(n)}u=u_{-n-1}\1$ for $n\in \BN,\ u\in V$. As $\fh$ generates $V_{\widehat\fh}(\ell,0)$ as a vertex algebra,
 it follows that a left ideal of $V_{\widehat\fh}(\ell,0)$ exactly amounts to an $\widehat{\fh}$-submodule.
Thus, an ideal of $V_{\widehat\fh}(\ell,0)$ amounts to an $\widehat{\fh}$-submodule which is also a $\B$-submodule.

The following is a technical result:

\begin{lemma}\label{th:Duin}
Let $u\in \fh,\ k, r\in \BN, \ n\in\BZ_+$.  Then
\begin{align}
    \D^{(k)}(u(-n)^p)&=0\  \   \   \text{ if }p\nmid k, \label{eq:Dkup02}\\
    \D^{(rp)}(u(-n)^p)&=\binom{n+r-1}{r}u(-n-r)^p,\label{eq:Dkup01}\\
      \D^{(k)}(u(-n))&=0\  \  \   \text{ if }p\mid n\  \text{ and } \  p\nmid k.\label{eq:Dkup04}
\end{align}
\end{lemma}

\begin{proof} Noticing that $\widehat{\fh}_{+}$ is abelian and that
$\D^{(k)}(\widehat{\fh}_{+})\subset \widehat{\fh}_{+}$ for $k\in \BN$,
we have
\begin{align*}
    e^{x\D}(u(-n)^{p})&=\left(e^{x\D}u(-n)\right)^{p}=\left(\sum_{r\ge 0}(-1)^{r}\binom{-n}{r}u(-n-r)x^{r}\right)^{p}\\
    &=\sum_{r\ge 0}\binom{n+r-1}{r}^{p}u(-n-r)^{p}x^{rp}.
\end{align*}
From this we obtain \eqref{eq:Dkup02} immediately.
On the other hand, using Fermat's little theorem, we get
\begin{eqnarray*}
    \D^{(rp)}(u(-n)^p)
    =\binom{n+r-1}{r}^p u(-n-r)^p=\binom{n+r-1}{r}u(-n-r)^p,
\end{eqnarray*}
proving \eqref{eq:Dkup01}. If $p\mid n$ and $p\nmid k$, we have $\binom{-n}{k}=0$, so that
\begin{equation*}
    \D^{(k)}(u(-n))=(-1)^k\binom{-n}{k}u(-n-k)=0,
\end{equation*}
proving \eqref{eq:Dkup04}.
\end{proof}

Let $\lambda\in  (\widehat\fh_{+})^{*}$.
Denote by $J(\ell,\lambda)$ the $\widehat\fh$-submodule of $V_{\widehat\fh}(\ell,0)$ generated by vectors
\begin{equation}\label{egenerators}
\begin{split}
    & (u(-m)^p-\lambda_{m} (u)^p)\1  \   \   \  \  \mbox{ for }u\in \fh,\ m\in \BZ_+, \mbox{ and }\\
    & (u(-n)-\lambda_{n}(u))\1   \     \    \   \   \mbox{ for } u\in \fh,\ n\in p\BZ_+.
\end{split}
\end{equation}
Recall that these vectors in
(\ref{egenerators}) lie in $\Omega_{V_{\widehat\fh}(\ell,0)}$.
Set
\begin{equation}
    L_{\widehat\fh}(\ell,0,\lambda)=V_{\widehat\fh}(\ell,0)/ J(\ell,\lambda),
\end{equation}
an $\widehat\fh$-module of level $\ell$.
From Proposition~\ref{th:L-irr}, $L_{\widehat\fh}(\ell,0,\lambda)$ is an irreducible $\widehat\fh$-module,
i.e., $J(\ell,\lambda)$ is a maximal $\widehat\fh$-submodule of $V_{\widehat\fh}(\ell,0)$.
In fact, these are all the maximal $\widehat\fh$-submodules.

\begin{lemma}
$J(\ell,\lambda)$ with $\lambda\in  (\widehat\fh_{+})^{*}$ exhaust the maximal $\widehat{\fh}$-submodules
of $V_{\widehat\fh}(\ell,0)$.
\end{lemma}

\begin{proof}  Assume that $K$ is a maximal  $\widehat\fh$-submodule of $V_{\widehat\fh}(\ell,0)$,
 so that the quotient module $V_{\widehat\fh}(\ell,0)/K$ is irreducible.
 As $V_{\widehat\fh}(\ell,0)$ is clearly of countable dimension, $V_{\widehat\fh}(\ell,0)/K$ is of countable dimension.
 By Schur lemma (which states that if $U$ is an irreducible module of countable dimension for an associative algebra $A$,
then any central element of $A$ acts on $U$ as a scalar),
 there exists $\lambda\in  (\widehat\fh_{+})^{*}$
such that for $u\in \fh,\ n\in \BZ_{+}$, $u(-n)^{p}$ acts on $V_{\widehat\fh}(\ell,0)/K$
as scalar $\lambda(u(-n))^{p}$ and $u(-np)$ acts as scalar $\lambda(u(-np))$.
It follows that those vectors in (\ref{egenerators}) with this particular $\lambda$ belong to $K$.
Thus $J(\ell,\lambda)\subset K$.  Since  $V_{\widehat\fh}(\ell,0)/ J(\ell,\lambda)\ \left(=L_{\widehat\fh}(\ell,0,\lambda)\right)$
is an irreducible $\widehat{\fh}$-module, we must have $K=J(\ell,\lambda)$.
\end{proof}

Furthermore, we have:

\begin{proposition}\label{th:leftideal}
$J(\ell,\lambda)$ is an ideal of $V_{\widehat\fh}(\ell,0)$ if and only if
$\lambda_{n} (u)=0$ for all $u\in \fh,\ n> 1$.
\end{proposition}

\begin{proof} Notice that as $\fh$ generates $V_{\widehat\fh}(\ell,0)$ as a vertex algebra,
a left ideal of $V_{\widehat\fh}(\ell,0)$ amounts to an $\widehat\fh$-submodule of $V_{\widehat\fh}(\ell,0)$.
Thus, $J(\ell,\lambda)$ is a left ideal of $V_{\widehat\fh}(\ell,0)$.
Assume $\lambda_{n}(u)=0$ for $u\in \fh,\ n>1$. Denote by $Q$ the linear span of the vectors in
(\ref{egenerators}).
Let $k\in \BN,\ u\in \fh,\ n\in p\BZ_+$. If $p\mid k$, with $\lambda_{n}(u)=\lambda_{n+k}(u)=0$ we have
$$\D^{(k)}(u(-n)-\lambda_{n}(u))\1
=(-1)^{k}\binom{-n}{k}\left(u(-n-k)-\lambda_{n+k}(u)\right)\1\in Q.$$
If $p\nmid k$,  as $\binom{-n}{k}=0$ we have
$$\D^{(k)}(u(-n)-\lambda_{n}(u))\1=(-1)^{k}\binom{-n}{k}u(-n-k)\1=0\in Q.$$
Now, let $k\in \BN,\ u\in \fh,\ n\in \BZ_+$.
If $p\nmid k$ (which implies $k\ne 0$), using \eqref{eq:Dkup02}, we have
$$\D^{(k)}(u(-n)^{p}-\lambda_{n}(u)^{p})\1=\D^{(k)}(u(-n)^{p})\1=0\in Q.$$
Assume $p\mid k$ with $k=rp$. By (\ref{eq:Dkup01}), we have
\begin{align*}
\D^{(k)}(u(-1)^{p}-\lambda_{1}(u)^{p})\1&=u(-1-r)^{p}\1-\delta_{r,0}\lambda_{1}(u)^{p}\1\\
&=u(-1-r)^{p}\1-\lambda_{1+r}(u)^{p}\1\in Q.
\end{align*}
For $n\ge 2$, we have
\begin{align*}
\D^{(k)}(u(-n)^{p}-\lambda_{n}(u)^{p})\1&=\D^{(k)}(u(-n)^{p})\1=\binom{n+r-1}{r} u(-n-r)^p\1\\
&=\binom{n+r-1}{r}\left( u(-n-r)^p-\lambda_{n+r}(u)^{p}\right)\1\in Q.
\end{align*}
This shows that $Q$ is a $\B$-submodule.
It then follows that $J(\ell,\lambda)$ is a $\B$-submodule. Therefore, $J(\ell,\lambda)$ is an ideal.

Conversely, suppose $J(\ell,\lambda)$ is an ideal.
Then $J(\ell,\lambda)$ is both an $\widehat\fh$-submodule and a $\B$-submodule.
Let $u\in \fh$ and assume $n\ge2$. By \eqref{eq:Dkup01}, we have
\begin{align*}
  J(\ell,\lambda)\ni \D^{((n-1)p)}(u(-1)^p\1-\lambda_{1}(u)^p\1)=u(-n)^p\1,
\end{align*}
which implies $\lambda_{n}(u)^{p}\1=u(-n)^p\1-\left(u(-n)^p-\lambda_{n}(u)^{p}\right)\1\in J(\ell,\lambda)$.
 Since
$J(\ell,\lambda)\ne V_{\widehat\fh}(\ell,0)$, we have $\1\notin J(\ell,\lambda)$.
Then we conclude that $\lambda_{n}(u)=0$, as desired.
\end{proof}

Define
 \begin{eqnarray}
 \Lambda=\{ \lambda\in (\widehat\fh_{+})^{*}\ \mid \  \lambda(u(-n))=0\  \  \mbox{ for }u\in \fh,\ n\ge 2\}.
 \end{eqnarray}
As an immediate consequence of Proposition \ref{th:leftideal}, we have:

\begin{corollary}\label{th:simple}
Let $\ell\in \BF^{\times}$ and $\lambda\in \Lambda$.
Then $L_{\widehat\fh}(\ell,0,\lambda)$ is an irreducible $\widehat\fh$-module
and a simple vertex algebra.
Moreover, $L_{\widehat\fh}(\ell,0,\lambda)$ with $\lambda\in  (\widehat\fh_{+})^{*}$
exhaust simple quotient vertex algebras of $V_{\widehat\fh}(\ell,0)$, which are also irreducible
$\widehat\fh$-modules.
\end{corollary}


Next, we determine all modules for vertex algebras $L_{\widehat\fh}(\ell,0,\lambda)$.
To this end, we need the following technical result (cf. \cite{DL}):

\begin{lemma}\label{th:a(x)_n^p}
Let $V$ be a vertex algebra and let $(W,Y_{W})$ be a $V$-module. Let $a\in V$ be such that
$a_{n}$, $n\in\BN$, mutually commute and
$a_{-n}$, $n\in\BZ_+$, mutually commute. Then, for $n\in\BZ_+$,
\begin{equation*}
    Y_{W}(a_{-n}^p\1,x)=\sum_{j\ge0}\binom{n+j-1}{j} a_{-n-j}^p x^{j p}
            +\sum_{j\ge0}(-1)^{1-n}\binom{n+j-1}{j} a_{j}^p x^{-p(n+j)}.
\end{equation*}
In particular, we have
\begin{equation*}
    Y_{W}(a_{-1}^p \1,x)=\sum_{j\in\BZ} a_{j}^p x^{-p(j+1)}.
\end{equation*}
\end{lemma}

\begin{proof}
Let $n\in\BZ_+$. Set
\begin{align*}
    A=\sum_{j\ge0}\binom{-n}{j}(-1)^j a_{-n-j}x^{j},\  \  \  \
    B=-\sum_{j\ge0}\binom{-n}{j}(-1)^{-n-j}a_{j}x^{-n-j}.
\end{align*}
We first use induction on $k$ to show that
\begin{equation}\label{eq:a(x)_n^p01}
    Y_{W}\left((a_{-n})^k \1,x\right) =\sum_{j=0}^k\binom{k}{j}A^{k-j}B^j
\end{equation}
for all $k\in \BZ_+$.
From the Jacobi identity, we have
$$Y_{W}(a_{-n}v,x)=AY_{W}(v,x)+Y_{W}(v,x)B$$
for $v\in V$. In particular,
$Y_{W}(a_{-n}\1,x)=A+B$. The induction step is given by
\begin{align*}
    Y_{W}\left((a_{-n})^{k+1}\1,x\right)&=Y_{W}\left(a_{-n}(a_{-n})^k \1,x\right)\\
    &=A\sum_{j=0}^k\binom{k}{j}A^{k-j}B^j +\sum_{j=0}^k\binom{k}{j}A^{k-j}B^j B \\
    &=\sum_{j=0}^k \binom{k}{j}A^{k+1-j}B^j +\sum_{j=0}^k \binom{k}{j}A^{k-j}B^{j+1} \\
    &=\sum_{j=0}^{k+1}\binom{k+1}{j} A^{k+1-j}B^j .
\end{align*}
Taking $k=p$ in \eqref{eq:a(x)_n^p01} we get $Y_{W}(a_{-n}^p\1,x) =A^p +B^p$ as $\binom{p}{j}=0$ for $0<j<p$.
Using Fermat's little theorem, we obtain
\begin{align*}
    Y_{W}(a_{-n}^p\1,x)&=A^p+B^p\\
    &=\sum_{j\ge0}(-1)^{j p}\binom{-n}{j}^p a_{-n-j}^p x^{j p}
            +\sum_{j\ge0}(-1)^{(1-n-j)p}\binom{-n}{j}^p a_{j}^p x^{-p(n+j)}\\
    &=\sum_{j\ge0}(-1)^{j}\binom{-n}{j} a_{-n-j}^p x^{j p}
            +\sum_{j\ge0}(-1)^{1-n-j}\binom{-n}{j} a_{j}^p x^{-p(n+j)}\\
    &=\sum_{j\ge0}\binom{n+j-1}{j} a_{-n-j}^p x^{j p}
            +\sum_{j\ge0}(-1)^{1-n}\binom{n+j-1}{j} a_{j}^p x^{-p(n+j)},
\end{align*}
as desired.
\end{proof}

The following is a key result for our goal:

\begin{proposition}\label{th:L-module-property}
Let $\ell\in \BF$ and $\lambda\in \Lambda$.
Suppose that $W$ is an $L_{\widehat\fh}(\ell,0,\lambda)$-module.
Then $W$ is a restricted $\widehat\fh$-module of level $\ell$, satisfying the following conditions
for $u\in \fh$:
\begin{enumerate}[(i)]
\item $u(n)$ act trivially on $W$ for $n\in p\BZ$;
\item $u(n)^p$ act trivially on $W$ for $n\in \BZ$ with $n\ne -1$;
\item $u(-1)^p$ acts on $W$ as scalar $\lambda(u(-1))^p$.
\end{enumerate}
\end{proposition}

\begin{proof} As $u(-p)\1 =\lambda(u(-p))\1=0$ in $L_{\widehat\fh}(\ell,0,\lambda)$, we have $Y_W(u(-p)\1,x)=0.$
On the other hand,  we have
\begin{equation}\label{eq:L-module-property01}
\begin{split}
    \quad Y_W(u(-p)\1,x)
    &=\sum_{j\ge0}\binom{-p}{j} (-1)^{j}\left(u(-p-j)x^j-(-1)^{p}u(j) x^{-p-j}\right)\\
    &=\sum_{j\ge0}\binom{p+j-1}{j}\left( u(-p-j) x^j+u(j) x^{-p-j}\right).
\end{split}
\end{equation}
Note that from Lucas' theorem, $\binom{p+j-1}{j}=1$ when $p\mid j$.
Then (i) holds.

Since $u(-1)^p\1=\lambda(u(-1))^p\1$ in $L_{\widehat\fh}(\ell,0,\lambda)$,
it follows from Lemma~\ref{th:a(x)_n^p} that
\begin{align*}
  \lambda(u(-1))^p 1_W&=Y_W(u(-1)^p\1,x) =\sum_{j\in\BZ} u(j)^p x^{(-j-1)p}.
\end{align*}
Then (ii) and (iii) follow immediately.
\end{proof}

Now, we are in a position to present the main result of the paper:

\begin{theorem}\label{th:main}
Let $\ell\in \BF^{\times}$ and $\lambda\in \Lambda$.
Then every $L_{\widehat\fh}(\ell,0,\lambda)$-module is completely reducible and the adjoint module
$L_{\widehat\fh}(\ell,0,\lambda)$ is
the only irreducible $L_{\widehat\fh}(\ell,0,\lambda)$-module up to equivalence.
\end{theorem}

\begin{proof}
Let $W$ be any  $L_{\widehat\fh}(\ell,0,\lambda)$-module.
By Proposition~\ref{th:L-module-property},
$W$ is a restricted $\widehat{\fh}$-module of level $\ell$, satisfying condition $\mathcal{C}_0$.
Then by Theorem~\ref{th:h-mod-reducibility}, $W$ as an $\widehat{\fh}$-module is completely reducible.
Note that an $L_{\widehat\fh}(\ell,0,\lambda)$-submodule of $W$ is the same as an $\widehat{\fh}$-submodule.
Thus $W$ is a completely reducible $L_{\widehat\fh}(\ell,0,\lambda)$-module.
On the other hand, assume that $W$ is an irreducible $L_{\widehat\fh}(\ell,0,\lambda)$-module.
 In view of Lemma~\ref{th:vacuum}, $W$ contains a vacuum vector $w$.
By Propositions \ref{th:L-module-property} and \ref{th:L-irr},
we have $U(\widehat{\fh})w\simeq L_{\widehat\fh}(\ell,0,\lambda)$ and
consequently,  $W\simeq L_{\widehat\fh}(\ell,0,\lambda)$.
\end{proof}

\begin{remark}
Recall that in the case of characteristic zero, a vertex operator algebra $V$ is said to be {\em rational} if
every $\BN$-graded (namely admissible) $V$-module is completely reducible (see \cite{ZHU}, \cite{DLM}), and $V$ is
said to be {\em holomorphic} if $V$ is rational and if the adjoint module $V$ is the only irreducible module up to equivalence.
We also recall that a vertex algebra $V$ is {\em regular} if every $V$-module is completely reducible (see \cite{DLM1}, \cite{DY}).
In view of Lemma~\ref{th:heisenbergVOA}, Corollary \ref{th:simple}, and Theorem \ref{th:main},
vertex algebras $L_{\widehat\fh}(\ell,0,\lambda)$ with $\ell\in \BF^{\times},\ \lambda\in \Lambda$
are rational and holomorphic in a certain sense.
\end{remark}

For the rest of this section, we discuss the notion of vertex operator algebra.
We assume that $\BF$ is an algebraically closed field of characteristic $p>2$.
The definition of the notion of vertex operator algebra requires a slight modification
(see \cite{FLM}, \cite{FHL}, \cite{DR1}, \cite{DR2}):

\begin{definition}\label{def-voa}
A \emph{vertex operator algebra} is a vertex algebra $(V,Y,\1)$,
equipped a $\BZ$-grading
\begin{equation*}
    V=\bigoplus_{n\in\BZ} V_{(n)}
\end{equation*}
such that $\dim V_{(n)}<\infty$ for all $n$ and $V_{(m)}=0$ for $m$ sufficiently small,
and equipped with a vector  $\omega\in V_{(2)}$, called a \emph{conformal vector}, such that
\begin{equation}
    [L(m),L(n)]=(m-n)L(m+n)+\frac{1}{2}\binom{m+1}{3}\delta_{m+n,0}c_V
\end{equation}
for $m,n\in\BZ$, where
\begin{equation}
    Y(\omega,x)=\sum_{n\in\BZ}L(n)x^{-n-2}\ \left(=\sum_{n\in\BZ}\omega_n x^{-n-1}\right),
\end{equation}
and $c_V\in \BF$ (\emph{central charge} or \emph{rank} of $V$), and such that
\begin{align*}
    &L(0)v=n v \quad\text{for }\  v\in V_{(n)},\ n\in\BZ,
\end{align*}
$$L(-1)=\D^{(1)}\     \mbox{ on }V,$$
\begin{equation}
    u_nv\in V_{(s+t-n-1)} \quad\text{for }u\in V_{(s)}, \  v\in V_{(t)}, \  n,s,t\in\BZ.
\end{equation}
\end{definition}

Note that  $\widehat\fh$ is a $\BZ$-graded Lie algebra with $\deg {\bf k}=0$ and $\deg (\fh\otimes t^{n})=-n$ for $n\in \BZ$.
Then  $V_{\widehat\fh}(\ell,0)$ is a $\BZ$-graded $\widehat\fh$-module with
$\deg \1=0$ and $\deg \fh(n)=-n$ for $n\in \BZ$.
As in the case of characteristic zero, using the Segal-Sugawara construction
we have:

\begin{lemma}\label{th:heisenbergVOA}
Let $\ell\in \BF^{\times}$.
Then the vertex algebra $V_{\widehat\fh}(\ell,0)$
is a vertex operator algebra of central charge $d=\dim \fh$
with the conformal vector given by
\begin{equation*}
    \omega=\frac{1}{2\ell}\sum_{i=1}^d u^{(i)}(-1)u^{(i)}(-1)\1,
\end{equation*}
where $\{u^{(1)},\ldots,u^{(d)}\}$ is any orthonormal basis of $\fh$.
Furthermore, $\fh= V_{\widehat\fh}(\ell,0)_{(1)}$
and $\fh$ generates $V_{\widehat\fh}(\ell,0)$ as a vertex algebra,
and
\begin{equation*}
    [L(m),a(n)]=-na(m+n)\quad\text{for }a\in\fh, \ m,n\in\BZ.
\end{equation*}
\end{lemma}

It can be readily seen that for $\ell\in \BF^{\times}$, the ideal $J(\ell,0)$  of $V_{\widehat{\fh}}(\ell,0)$
is  $\BZ$-graded. Then we immediately have:

\begin{corollary}
For any $\ell \in \BF^{\times}$,  the quotient vertex algebra
$L_{\widehat{\fh}}(\ell,0,0)$ is a simple vertex operator algebra.
\end{corollary}

\begin{remark}
Let $\ell\in \BF^{\times},\ \lambda\in \Lambda$ with $\lambda\ne 0$.
We see that the ideal $J(\ell,\lambda)$ is {\em not} $\BZ$-graded. In view of this, the quotient vertex algebra
$L_{\widehat{\fh}}(\ell,0,\lambda)$ is {\em not} a vertex operator algebra
in the sense of Definition \ref{def-voa}.
On the other hand, for $u\in \fh,\ n\in \BZ_{+}$, we have
$$L(0)u(-np){\bf 1}=(np)u(-np){\bf 1}=0,\   \   \   \  L(0)u(-n)^{p}{\bf 1}=(np)u(-n)^{p}{\bf 1}=0.$$
It follows that the ideal $J(\ell,\lambda)$ of $V_{\widehat{\fh}}(\ell,0)$ is $L(0)$-stable.
Then $L_{\widehat{\fh}}(\ell,0,\lambda)$ is  $\BZ_{p}$-graded
by the eigenvalues of $L(0)$, where the $L(0)$-eigenspaces are infinite-dimensional.
In view of this,
$L_{\widehat{\fh}}(\ell,0,\lambda)$ is a conformal vertex algebra in a certain sense.
\end{remark}

\end{document}